\documentclass[pdflatex,sn-mathphys-num]{sn-jnl}

\usepackage{graphicx}%
\usepackage{multirow}%
\usepackage{amsmath,amssymb,amsfonts}%
\usepackage{amsthm}%
\usepackage{mathrsfs}%
\usepackage[title]{appendix}%
\usepackage{xcolor}%
\usepackage{textcomp}%
\usepackage{manyfoot}%
\usepackage{booktabs}%
\usepackage{algorithm}%
\usepackage{algorithmicx}%
\usepackage{algpseudocode}%
\usepackage{listings}%
\usepackage{natbib}
\usepackage{mathtools}
\usepackage[shortlabels]{enumitem}
\setlist{font=\normalfont}

\usepackage{tikz}
\usetikzlibrary{positioning}
\usetikzlibrary{cd}
\usetikzlibrary{arrows}
\usetikzlibrary{arrows.meta,
                decorations.markings}
\usetikzlibrary{shapes}

\DeclareMathOperator{\ord}{ord}
\DeclareMathOperator{\lcm}{lcm}
\DeclareMathOperator{\wval}{val}
\DeclareMathOperator{\fraction}{q}
\DeclareMathOperator{\denom}{h}
\DeclareMathOperator{\numer}{a}
\DeclareMathOperator{\rem}{r}
\DeclareMathOperator{\cl}{cl}
\DeclareMathOperator{\clz}{cl_0}

\newcommand{\WordsetD}[1]{{W}_{#1}}
\newcommand{\rQuoth}[2]{{Q}({#1},{#2})}
\newcommand{\rWordh}[2]{{W}({#1},{#2})}
\newcommand{\latn}{L_n}
\newcommand{\Z}{\mathbb{Z}}
\newcommand{\N}{\mathbb{N}}
\newcommand{\R}{\mathbb{R}}
\newcommand{\Q}{\mathbb{Q}}
\newcommand{\abs}[1]{\left\vert #1 \right\vert}

\theoremstyle{thmstyleone}%
\newtheorem{theorem}{Theorem}[section]
\newtheorem{proposition}[theorem]{Proposition}%
\newtheorem{corollary}[theorem]{Corollary}%
\newtheorem{fact}[theorem]{Fact}
\newtheorem{lemma}[theorem]{Lemma}

\theoremstyle{thmstyletwo}%
\newtheorem{example}[theorem]{Example}%
\newtheorem{remark}[theorem]{Remark}%

\theoremstyle{thmstylethree}%
\newtheorem{definition}[theorem]{Definition}%

\numberwithin{equation}{section}

\begin{document}

\title[Words don't come easy: strong affine representations of the polycyclic monoids]{Words don't come easy: strong affine representations of the polycyclic monoids}


\author*[1]{\fnm{Kristóf} \sur{Varga}}\email{vargak@server.math.u-szeged.hu}

\author[1]{\fnm{Tamás} \sur{Waldhauser}}\email{twaldha@math.u-szeged.hu}


\affil[1]{\orgdiv{Bolyai Institute}, \orgname{University of Szeged}, \orgaddress{\street{Aradi vértanúk tere 1.}, \city{Szeged}, \postcode{6720}, 
\country{Hungary}}}




\abstract{Each one-dimensional strong affine representation of the polycyclic monoid $\mathcal{P}_n$ is induced by a complete system of residues modulo $n$.
We completely characterize these representations in the case when the system of residues is an arithmetic sequence.
We accomplish this by introducing a closure operator on the set of primitive words over $\{0,1,\dots,n-1\}$, and we also describe the lattice of closed sets. 
This characterization covers all one-dimensional strong affine representations of $\mathcal{P}_2$.
}




\maketitle

\section{Introduction}\label{sec1}



In this paper we characterize certain kinds of representations of the polycyclic monoids.
The \emph{polycyclic monoid} $\mathcal{P}_{n}$ is a monoid with zero given by the presentation
\[
\mathcal{P}_{n}=\langle a_{0},\ldots,a_{n-1},a_{0}^{-1},\ldots,a_{n-1}^{-1}%
:a_{i}^{-1}a_{i}=1\text{ and }a_{i}^{-1}a_{j}=0 \text{ if } i\neq j\rangle.
\]
These monoids were defined by Nivat and Perrot \cite{NP} as generalizations of the bicyclic monoid (actually, $\mathcal{P}_{1}$ is the bicyclic monoid extended by a zero element).
The polycyclic monoids are sometimes also referred to as Cuntz inverse semigroups, since they were also discovered by Cuntz in the context of $C^*$ algebras \cite{C}.

A \emph{representation} of $\mathcal{P}_n$ is defined as a monoid homomorphism from $\mathcal{P}_n$ to the monoid $\mathcal{I}_X$ of all partial bijections on a set $X$. 
In this paper we study so-called \emph{strong affine representations}. 
These are given by $X=\mathbb{Z}^{\nu}$ and $f_{i}\left(
\mathbf{x}\right)  =N\mathbf{x}+\mathbf{d}_{i}$, where $N$ is a nonsingular
$\nu\times\nu$ matrix over $\mathbb{Z}$ with $\abs{\det N}=n$, and $\mathbf{d}_{0},\dots
,\mathbf{d}_{n-1}$ is a complete system of representatives of the cosets corresponding to
the subgroup $N\mathbb{Z}^{\nu}$ of $\mathbb{Z}^{\nu}$.
The representation is then the unique monoid homomorphism $\mathcal{P}_n \to \mathcal{I}_{\Z^{\nu}}$ mapping $a_i$ to $f_i$ and $a_i^{-1}$ to $f_i^{-1}$.

In the one-dimensional case, previous studies considered the case $\abs{\det N}>0$, i.e., representations of $\mathcal{P}_{n}$ for $n \geq 2$ that assign to each generator $a_i$ an affine map $f_i\colon \Z \to \Z,~x \mapsto nx+d_i$, where $D=\{d_0,\dots,d_{n-1}\}$ is a complete system of residues modulo $n$.
Such a representation can be visualized by an edge-labeled directed graph, where an edge with label $i$ is drawn from $x$ to $f_i(x)$ for each $x \in \Z$ and $i \in \{ 0,\dots,n-1 \}$.
Two such representations are equivalent if and only if the corresponding graphs are isomorphic (here the isomorphism is required to preserve the edge labels).

The graph of the representation has finitely many connected components and every component contains a cycle; the vertices of these cycles are called periodic points.
Bratteli and Jørgensen \cite{BJ} investigated the cycles of these representations (see Section~\ref{sec:prelim} for the precise definitions), and provided some concrete examples of such cycle structures.

Consider a cycle $C$ of length $\ell$, and let us read the labels of the edges of $C$ starting from a vertex $x$, following the arrows backward. This way we obtain a primitive word $w$ of length $\ell$ over the alphabet $\{0,1,\dots,n-1\}$. 
We will say that $w$ is the word associated with the periodic point $x$.
If we start reading the labels from another vertex of $C$, then we obtain a cyclic shift of $w$, hence the word corresponding to a cycle is determined only up to cyclic shifts. 
Therefore, every one-dimensional strong affine representation of $\mathcal{P}_n$ can be given by such a finite set of words \cite{HW}.

Jones and Lawson \cite{JL} initiated the systematic study of these sets of words, and for $n=2$ they determined the words of a representation in terms of binary expansions of certain fractions that can be computed from $D$.
The results of \cite{HW} indicate that for $n>2$ the words (or even the cycles) corresponding to a representation depend on $D$ in a very complicated way, thus it does not seem feasible to describe these words in full generality. Therefore, we restrict our attention to the case when $D$ is an arithmetic sequence.
In this case $n$-ary expansions of fractions can be used, just like for $n=2$ \cite{HW} (note that any two numbers form an arithmetic sequence).
It was also shown in \cite{HW} that any finite set of primitive words can be extended to the set of words of a representation of this ``arithmetic" type.

We prove that there is always a least one among such extensions, thus we obtain a closure operator on the set of primitive words over $\{0,1,\dots,n-1\}$.
Our main result is a concrete description of this closure operator.
We determine the lattice of closed sets and we also provide a method to compute the closure of a set of words.
Earlier investigations were restricted to positive values of $n$, but we also allow $n \leq -2$.
Therefore, throughout the paper we assume that $\abs{n} \geq 2$, and we aim at describing representations of $\mathcal{P}_{\abs{n}}$ in terms of words over the alphabet $\Sigma := \{0,1,\dots,\abs{n}-1\}$.

The structure of the paper is as follows.
In Section~\ref{sec:expansions} we recall some necessary facts about $n$-ary expansions of real and rational numbers that are well known for positive $n$, and we present the less known counterparts of these facts for negative values of $n$.
In Section~\ref{sec:one-dim_sarpm} we summarize results about the dynamical properties of one-dimensional strong affine representations and their relationship with $n$-ary expansions of certain fractions.
These have been proved for positive $n$ in \cite{HW}; some proofs are the same in the case of negative $n$, while some proofs need modifications.
The main results are contained in Section~\ref{sec:arithmetic}: using the connection between representations and fractions we prove that the sets of words corresponding to representations induced by arithmetic sequences form a closure system. 
We describe the join operation of the lattice of closed sets, and we show how this can be used to compute the closure of a set of words.
The closure operator becomes simpler if we restrict the arithmetic sequences to ones that start with zero (i.e., $d_0=0$); in Section~\ref{sec:d_0=0} we use this simplified form to define a quasiorder relation on the primitive words over $\Sigma$ such that the closed sets are exactly the principal ideals of this quasiorder.
Obviously, any two numbers form an arithmetic sequence; thus our results cover all one-dimensional strong affine representations of $\mathcal{P}_2$. 
Moreover, if $n = 2$, then one can assume without loss of generality that $d_0=0$, thus in this case one can use the quasiorder defined in Section~\ref{sec:d_0=0} instead of the closure operator. Finally, in Section~\ref{sec:conclusion} we prove a curious fact that relates Mersenne primes with representations of $\mathcal{P}_2$, and we formulate some open questions motivating further research in this area.

\section{Expansions of real numbers in positive and negative bases}\label{sec:expansions}

Let $n$ be an integer \emph{base} with $\abs{n} \geq 2$, and let $\Sigma := \{0,1,\dots,\abs{n}-1\}$ denote the set of \emph{digits}.
By an \emph{$n$-ary expansion} of a real number $z$, we mean a sum of the form $z = \sum_{i=-\infty}^m c_in^i$ with $c_i\in \Sigma$.
Using a ``decimal point'', we can write this $n$-ary expansion as $z = (c_{m}\ldots c_{0}.c_{-1}c_{-2}\ldots)_n$.
Note that we allow zero digits on the left (hence we can assume $m \geq 0$ without loss of generality and we do not distinguish between expansions differing only in leading zeros) and the base $n$ is allowed to be negative. 
If all but finitely many of the coefficients $c_i$ are zeros, then we speak about a \emph{finite $n$-ary expansion}.

The case $n \geq 2$ is well known, while negative bases were considered e.g., in \cite{DWM} in more generality with arbitrary ``digit sets'' in place of $\Sigma$.
In this section we recall some results about $n$-ary expansions that will be needed in the sequel.

\subsection{Expansions of real numbers}

First we consider the existence and (non)uniqueness of $n$-ary expansions of real numbers.

\begin{proposition}[\cite{DWM}]\label{p:radix representations of integers}
	Let $n$ be an integer base with $\abs{n} \geq 2$.
	\begin{enumerate}[(i)]
		\item If $n$ is positive, then every nonnegative integer has a finite $n$-ary expansion, and this finite expansion is unique.
		\item If $n$ is negative, then every integer has a finite $n$-ary expansion, and this finite expansion is unique.
	\end{enumerate}
\end{proposition}

\begin{remark}\label{r:basic}
	As the proposition above shows, negative bases are ``better'' than positive ones, since we can represent all integers (without using a minus sign).
	Following the terminology of \cite{DWM}, this means that the standard digit set $\Sigma$ is \emph{basic} if $n \leq -2$.
\end{remark}

\begin{remark}
	Strong affine representations of the polycyclic monoids give rise to $n$-ary expansions of a somewhat different kind as well, where negative numbers can be also represented in a positive base by \emph{formally} evaluating divergent geometric series, such as $-1 = \sum_{i=0}^\infty (n-1) n^i$ \cite[Remark 2.10]{HW}.
    This will not be used in this paper, but it is worth mentioning.
\end{remark}

Allowing infinite expansions, some real numbers may have more than one $n$-ary expansion.
Again, the positive case is well known (but we recall it in the next proposition), while for negative bases the results of \cite{DWM} imply that a real number can have at most three $n$-ary expansions.
We prove that actually only at most two $n$-ary expansions can yield the same number, and we explicitly describe the numbers with two expansions.

\begin{proposition}\label{p:radix representations of reals}
	Let $n$ be an integer base with $\abs{n} \geq 2$ and let $d = \abs{n}-1$.
	\begin{enumerate}[(i)]
		\item If $n$ is positive, then every nonnegative real number has an $n$-ary expansion.
		A number $z \in \R^+$ has multiple $n$-ary expansions if and only if it can be written as $z = \frac{a}{n^k}$ for some $a\in\N$ and $k\in\N_0$. 
		These numbers have exactly two $n$-ary expansions; one of them is finite, i.e., ending with $0000\dots$, the other one is infinite, ending with $dddd\dots$.
		\item If $n$ is negative, then every real number has an $n$-ary expansion.
		A number $z \in \R$ has multiple $n$-ary expansions if and only if it can be written as
		\[
			z=\frac{a}{n^k}-\frac{1}{n^k(n-1)}
		\]
		for some $a\in\Z$ and $k\in\N$.
		These numbers have exactly two $n$-ary expansions, both of them are infinite, one of them ending with $d0d0d0\dots$, the other one ending with $0d0d0d\dots$.
	\end{enumerate} 
\end{proposition}
\begin{proof}
	The first statement is common knowledge, so we assume that $n \leq -2$.
	According to \cite[Theorem 10]{DWM}, the existence of an $n$-ary expansion follows from the \emph{basic} property mentioned in Remark~\ref{r:basic}.
	
	It remains to prove the claim on numbers with several $n$-ary expansions. 
	To this extent, let us consider the set of real numbers having an $n$-ary expansion of the form $(c_{m}\ldots c_{0}.c_{-1}c_{-2}\ldots c_{-k}\ldots)_n$, where the first $m+1+k$ digits are fixed as indicated, and the digits after $c_{-k}$ are arbitrary.
	Let $I(a,k)$ denote the set of such real numbers, where $a = (c_{m}\ldots c_{0}c_{-1}c_{-2}\ldots c_{-k})_n \in \Z$.
	
	Assume first that $k$ is even.
	Then the least and greatest elements of $I(a,k)$ are 
	\begin{align}
		\min I(a,k) &= (c_{m}\ldots c_{0}.c_{-1}c_{-2}\ldots c_{-k}d0d0d0\ldots)_n = \frac{a}{n^k} - \frac{n}{n^k(n-1)},\label{e:min I(a,k)}\\
		\max I(a,k) &= (c_{m}\ldots c_{0}.c_{-1}c_{-2}\ldots c_{-k}0d0d0d\ldots)_n = \frac{a}{n^k} - \frac{1}{n^k(n-1)}\label{e:max I(a,k)}.
	\end{align}
	Observe that $\max I(a,k) = \min I(a+1,k)$, thus for a fixed $k$, the sets $I(a,k)$ can only meet at their extrema.
	Since every real number has an $n$-ary expansion, we have $\bigcup_{a \in \Z}I(a,k) = \R$ for every $k\in \N_0$, and this implies that each set $I(a,k)$ is an interval (of length $1/\abs{n}^k$).
	
	If $k$ is odd, then a similar argument shows that $I(a,k)$ is an interval, but in this case \eqref{e:min I(a,k)} gives the right endpoint and \eqref{e:max I(a,k)} gives the left endpoint.
    Again, for a fixed $k$, the intervals $I(\,\cdot\,,k)$ constitute a tiling of the real line.
    
    Note also that for all integers $k$, the interval $I(a,k)$ is a union of $\abs{n}$ intervals of the form $I(\,\cdot\,,k+1)$, which, of course, also meet only at their endpoints.
	
	Now let $z$ be a real number with an $n$-ary expansion $z = (c_{m}\ldots c_{0}.c_{-1}c_{-2}\ldots)_n$.
	Then $z \in I(a_k,k)$ for each $k \in \N_0$, where $a_k = (c_{m}\ldots c_{0}c_{-1}c_{-2}\ldots c_{-k})_n$.
	These intervals form a descending chain $I(a_0,0) \supset I(a_1,1) \supset \cdots$, and the only element of their intersection is $z$, as the lengths of the intervals converge to zero.
	If  $z$ has another $n$-ary expansion $z = (c'_{m'}\ldots c'_{0}.c'_{-1}c'_{-2}\ldots)_n$, then we obtain a similar nested chain of intervals $I(a'_k,k)$, also containing $z$.
	Since the two expansions (and hence the two chains of intervals) are different, there is a $k$ such that $I(a_k,k) \neq I(a'_k,k)$.
	By the tiling properties discussed above, $z$ can belong to both of these intervals only if either $z = \max I(a_k,k) = \min I(a'_k,k)$ or $z = \min I(a_k,k) = \max I(a'_k,k)$.
	This, together with \eqref{e:min I(a,k)} and \eqref{e:max I(a,k)} proves that $z$ is indeed of the form stated in the second statement of the theorem.
	The above argument also shows that $z$ cannot have a third $n$-ary expansion.
\end{proof}

\begin{corollary}\label{c:unit interval}
	Let $n$ be an integer base with $\abs{n} \geq 2$.
	Let $I_n$ denote the set of real numbers having an $n$-ary expansion of the form $(0.c_{-1}c_{-2}\ldots)_n$ with arbitrary digits after the ``decimal point''.
	\begin{enumerate}[(i)]
		\item If $n$ is positive, then $I_n = [0,1]$.
		\item If $n$ is negative, then $I_n = \left[\frac{-n}{n-1},\frac{-1}{n-1}\right]$.
	\end{enumerate} 
\end{corollary}
\begin{proof}
	The case of a positive base is well known, so we assume that $n \leq -2$.
	Using the notation $I(a,k)$ introduced in the proof of Proposition~\ref{p:radix representations of reals}, we have $I_n=I(0,0)$, hence the endpoints of the interval $I_n$ can be read from \eqref{e:min I(a,k)} and \eqref{e:max I(a,k)}.
\end{proof}

\begin{example}\label{ex:endpoints}
	If $n$ is positive, then the number $1$ has the following two $n$-ary expansions:
	\[
		1=(1.0000\dots)_n=(0.dddd\dots)_n.   
	\]
	In the usual decimal system this reads as follows:
	\[
		1=(1.0000\dots)_{10}=(0.9999\dots)_{10}.
	\]
	If $n$ is negative, then one of the simplest examples for a number with two $n$-ary expansions is
	\[
		\frac{-1}{n-1}=(0.0d0d0d\ldots)_n=(1.d0d0d0\ldots)_n.
	\]
	For $n=-10$, this reads as follows:
	\[
		\frac{1}{11}=(0.090909\ldots)_{-10}=(1.909090\ldots)_{-10}.
	\]
	Likewise, the left endpoint of the interval $I_n$ has the following two $n$-ary expansions for $n \leq -2$:
	\[
		\frac{-n}{n-1}=(1d.0d0d0d\ldots)_n=(0.d0d0d0\ldots)_n.
	\]
\end{example}

\subsection{Expansions of rational numbers}

The following proposition summarizes what we need to know about periodic $n$-ary expansions (we regard finite expansions as periodic). These facts are well known for positive bases and can be proved in exactly the same way for negative bases. Note that, according to Proposition~\ref{p:radix representations of reals}, we must assume that $z \geq 0$ in the second and third items if $n$ is positive, while $z$ can be arbitrary if $n$ is negative.

\begin{proposition}\label{p:rationalperiodic}
	Let $n$ be an integer base with $\abs{n} \geq 2$.
	\begin{enumerate}[(i)]
		\item If a real number has a periodic $n$-ary expansion, then it is a rational number.
		\item Conversely, if $z$ is a rational number, then all $n$-ary expansions of $z$ are periodic.
		\item Moreover, if $z = a/h$, where $h$ is relatively prime to $n$, then all $n$-ary expansions of $z$ are \emph{purely} periodic, i.e., the periodic part begins right after the ``decimal point''.
	\end{enumerate} 
\end{proposition}

If $z = (c_{m}\ldots c_{0}.c_{-1}c_{-2}\ldots)_n$ is a purely periodic $n$-ary expansion, then there is a least positive integer $\ell$, such that $c_i=c_{i-\ell}$ for all $i\leq -1$.
Thus $c_{-1}c_{-2}\ldots c_{-\ell}$ is the shortest repeating string of digits, called the \emph{repetend} of the expansion.
We will regard this repetend as a word over $\Sigma$, and the minimality of $\ell$ implies that the repetend is always a so-called primitive word (i.e., it is not a power of a shorter word).

In the next remark we take a closer look at the rational numbers in $I_n$ that have purely periodic $n$-ary expansions. 

\begin{remark}\label{r:repetend}
Recall from Corollary~\ref{c:unit interval} that the set of real numbers having an $n$-ary expansion of the form $(0.c_{-1}c_{-2}\ldots)$ is $I_n = [0,1]$ if $n \geq 2$ and $I_n = \left[\frac{-n}{n-1},\frac{-1}{n-1}\right]$ if $n\leq-2$.
Note that the length of the interval is $1$ in both cases.
It follows from Proposition~\ref{p:radix representations of reals} that if $z \in \Q$ is an interior point of $I_n$ and the denominator of the reduced form of $z$ is relatively prime to $n$, then $z$ has only one $n$-ary expansion (cf. also Proposition~\ref{p:rationalperiodic}).
On the other hand, if $n$ is positive, then the right endpoint of $I_n$ has two $n$-ary expansions, while if $n$ is negative, then both endpoints of $I_n$ have two $n$-ary expansions.
By Example~\ref{ex:endpoints}, in all of these cases only one of the expansions is of the form $(0.c_{-1}c_{-2}\dots)_n$, and we shall always consider this expansion.
By this convention, we have chosen a unique repetend for each purely periodic rational number $z$ from $I_n$, thus we can speak about the repetend of $z$ instead of the repetend of its expansion.
\end{remark}

The next proposition shows that the length of the repetend of a purely periodic reduced fraction $\frac ah$ almost always coincides with the order of $n$ modulo $h$.

\begin{proposition}\label{p:ordhn is fraction rep length}
    Let $h\in\N$ such that $\gcd(h,n)=1$ and let $a\in\Z$ such that $\gcd(a,h)=1$ and $\frac{a}{h}\in I_n$.
    \begin{enumerate}[(i)]
        \item If $n\geq2$, then the length of the repetend of $\frac ah$ is $\ord_h(n)$.
        \item If $n\leq-2$ and $\frac{a}{h}\notin\{\frac{-n}{n-1},\frac{-1}{n-1}\}$, then the length of the repetend of $\frac ah$ is $\ord_h(n)$.
    \end{enumerate}
\end{proposition}
\begin{proof}
    Since $\gcd(a,h)=1$ and $\frac{a}{h}\in I_n$, the $n$-ary expansion $\frac{a}{h} = (0.c_1c_2\dots)_n$ is purely periodic, i.e., $c_i=c_{i+\ell}$ for all $i \in \N$, where $\ell$ is the length of the repetend $c_1\dots c_\ell$.
    We consider the following four statements for a positive integer $k$:
    \begin{enumerate}[(a)]
        \item\label{i:k-per} $\forall i \in \N\colon c_i=c_{i+k}$;
        \item\label{i:egesz} $n^{k}\frac ah-\frac ah \in \Z$;
        \item\label{i:oszthato} $h \mid n^k-1$;
        \item\label{i:rend} $\ord_h(n) \mid k$.
    \end{enumerate}
    It is clear that \ref{i:rend} and \ref{i:oszthato} are equivalent, and from $\gcd(a,h)=1$ it is also clear that \ref{i:oszthato} is equivalent to \ref{i:egesz}.
    We shall prove that the first statement is ``almost" equivalent to the others.
    Let us write out the difference in \ref{i:egesz}:    
    \begin{align*}
    n^{k}\frac ah-\frac ah&=(c_1 \dots c_k.c_{k+1}\dots)_n-(0.c_1c_2\dots)_n \\
    &=(c_1\dots c_k)_n+(0.c_{k+1}c_{k+2}\dots)_n-(0.c_1c_2\dots)_n.
    \end{align*}
    Since $(c_1\dots c_k)_n\in\Z$, we see that \ref{i:egesz} holds if and only if 
    \begin{equation*}
    \beta := (0.c_{k+1}c_{k+2}\dots)_n-(0.c_1c_2\dots)_n\in\Z.
    \end{equation*}
    Observe that $\beta$ is the difference of two members of $I_n$ and $I_n$ is an interval of length $1$; therefore, $\beta \in \{0,1,-1\}$.
    By Remark~\ref{r:repetend}, the case $\beta=0$ is equivalent to \ref{i:k-per}.
    The cases $\beta=\pm 1$ occur if and only if $n \leq -2$ and $\frac{a}{h}$ is one of the endpoints of $I_n$.

    We have proved that the four statements are equivalent except when $n\leq-2$ and $\frac{a}{h}\in\{\frac{-n}{n-1},\frac{-1}{n-1}\}$.
    The least $k$ satisfying \ref{i:k-per} is the length of the repetend of $\frac{a}{h}$ and the least $k$ satisfying \ref{i:rend} is $\ord_h(n)$, thus the proposition is proved.
\end{proof}

\section{Cycles and words of one-dimensional strong affine representations}\label{sec:one-dim_sarpm}

The goal of this section is to summarize some basic facts about one-dimensional strong affine representations of the polycyclic monoids. 
Many of these results have been proved for positive $n$ in \cite{HW}; we will prove them for negative $n$ when it is necessary to modify the original proof.

\subsection{Notation}\label{sec:prelim}

We use the terminology and notation of \cite{HW}, and we refer the reader to that paper for more details and for further references.

From now on $n$ denotes an integer with $\abs{n} \geq 2$ (note that $n$ can be negative), and we are considering representations of $\mathcal{P}_{\abs{n}}$.
These representations will be described in terms of words over the alphabet $\Sigma = \{0,1,\dots,\abs n-1\}$.
We will use the standard notation $\Sigma^\ell$ to denote the set of words of length $\ell$, and $\Sigma^+$ will be used for the set of finite nonempty words over $\Sigma$.
A word $w\in \Sigma^\ell$ is \emph{primitive} if for every $k\in\N$ and $v\in \Sigma^+$, the equality $v^k=w$ implies $k=1$ and $v=w$. We say that $w_1$ and $w_2$ are \emph{conjugate} if $w_1$ can be obtained from $w_2$ using cyclic shifts. It is clear that conjugacy is an equivalence relation on $\Sigma^+$.
We denote the set of primitive words in $\Sigma^+$ by $\mathcal{W}$.

Let $D=\{d_0,\dots,d_{\abs{n}-1}\}$ be a complete system of residues modulo $n$. We will always assume $d_0<d_1<\dots<d_{\abs{n}-1}$.
The partial bijections $f_i\colon \Z \to \Z,~x \mapsto nx+d_i$ provide a one-dimensional strong affine representation of the polycyclic monoid $\mathcal{P}_{\abs{n}}$.

\subsection{Cycles and words of representations}

The union of the maps $f_i^{-1}~(i=0,\dots,\abs{n}-1)$ gives a map $R$ on $\Z$ given by
$R(x) = (x-d_i)/n$, where $d_i$ is the unique member of $D$ satisfying $x \equiv d_i~\pmod{n}$.
We will see in Corollary~\ref{c:periodic} that the representation is uniquely determined by the cycles of the map $R$.
As the following fact shows, we can always assume that $0\leq d_0\leq\abs{n-1}-1$.
\begin{fact}[{\cite[Fact~2.2]{HW}}]\label{f:+(n-1)}
    The following sequences give rise to equivalent representations of $\mathcal{P}_{\abs{n}}$:
    \begin{enumerate}[a)]
        \item\label{item: Da} $d_0,\dots,d_{\abs{n}-1}$;
        \item\label{item: Db} $d_0+k(n-1),\dots,d_{\abs{n}-1}+k(n-1)$ for arbitrary $k\in\Z$.
    \end{enumerate}
\end{fact}

\begin{proposition}\label{p:absorbing}
    Let $\abs{n} \geq 2$ and let $D = \{ d_0,\dots,d_{\abs{n}-1} \}$ be a complete system of residues modulo $n$ such that $d_0 < d_1 < \dots  < d_{\abs{n}-1}$. 
    \begin{enumerate}[(i)]
		\item\label{item:absorbing-positive} If $n$ is positive, then the following set is a finite positively invariant absorbing set for $R$:
		\[
			\mathcal{A}_n(D):=\left[\frac{-d_{n-1}}{n-1},\frac{-d_0}{n-1}\right] \cap \Z.
		\]
		\item\label{item:absorbing-negative} If $n$ is negative, then the following set is a finite positively invariant absorbing set for $R$:
		\[
			\mathcal{A}_n(D):=\left[\frac{-d_0n-d_{\abs{n}-1}}{n^2-1},\frac{-d_{\abs{n}-1}n-d_0}{n^2-1}\right] \cap \Z.
		\]
	\end{enumerate}
\end{proposition}
\begin{proof}
	For positive $n$, the result was proved in \cite[Lemma~2.3]{HW}, so we prove only \ref{item:absorbing-negative}.
	Assume thus that $n \leq -2$, and observe that this implies    
	\[y\leq x\leq z \implies \frac{z-d_0}n\leq\frac{x-d_0}n\leq R(x)\leq\frac{x-d_{\abs n -1}}{n}\leq\frac{y-d_{\abs n -1}}{n}\]
	for all $x,y,z \in \Z$.
    Letting $y$ and $z$ be the left and right endpoints of the interval in \ref{item:absorbing-negative}, we obtain the following inequalities for all $x\in \mathcal{A}_n(D)$
	\[\frac{-d_0n^2-d_{\abs{n}-1}n}{(n^2-1)n}=
	\frac{\frac{-d_{\abs{n}-1}n-d_0}{n^2-1}-d_0}{n}\leq
	R(x)\leq
	\frac{\frac{-d_0n-d_{\abs{n}-1}}{n^2-1}-d_{\abs n -1}}{n}=\frac{-d_{\abs{n}-1}n^2-d_0n}{(n^2-1)n}.
	\]
	Thus $R(x)\in \mathcal{A}_n(D)$, which means that $\mathcal{A}_n(D)$ is positively invariant.
	
	Now let us consider the set $D':=\{ nd_i+d_j : 0 \leq i,j \leq \abs{n}-1 \}$.
	This is a complete system of residues modulo $n^2$, hence the maps $x \mapsto n^2x+nd_i+d_j$ determine a representation of $\mathcal{P}_{n^2}$.
	The key observation is that the union of the inverses of these $n^2$ maps is nothing else but $R^2$.
	Since $n^2$ is positive, \ref{item:absorbing-positive} shows that $\mathcal{A}_{n^2}(D')$ is an absorbing set for $R^2$, thus it is also an absorbing set for $R$ (as the trajectory of $x$ under $R$ is the union of the trajectories of $x$ and $R(x)$ under $R^2$).
	To determine the set $\mathcal{A}_{n^2}(D')$, we need to find the least and the largest element of the set $D'$.
	Taking into account that $n$ is negative and the numbers $d_i$ are nonnegative, we have $\min D' = nd_{\abs{n}-1}+d_0$ and $\max D' = nd_0+d_{\abs{n}-1}$; therefore, \ref{item:absorbing-positive} gives
	\[
		\mathcal{A}_{n^2}(D') = \left[\frac{-\max D'}{n^2-1},\frac{-\min D'}{n^2-1}\right] \cap \Z
		= \left[\frac{-d_0n-d_{\abs{n}-1}}{n^2-1},\frac{-d_{\abs{n}-1}n-d_0}{n^2-1}\right] \cap \Z.
	\]
	This completes the proof, since this set is exactly the same as the one claimed in \ref{item:absorbing-negative}.
\end{proof}

The properties of one-dimensional strong affine representations of the polycyclic monoids summarized in the next corollary were proved for positive $n$ in \cite{HW}.
Having established the existence of a finite positively invariant absorbing set $\mathcal{A}_n(D)$ in Proposition~\ref{p:absorbing}, these results follow by the same arguments.

\begin{corollary}\label{c:periodic}
	Let $\abs{n} \geq 2$ and let $D = \{ d_0,\dots,d_{\abs{n}-1} \}$ be a complete system of residues modulo $n$ such that $d_0 < d_1 < \dots  < d_{\abs{n}-1}$. 
	The graph corresponding to the representation of $\mathcal{P}_{\abs{n}}$ determined by $D$ has finitely many connected components, and each component contains exactly one cycle, which are all contained in $\mathcal{A}_n(D)$.
	The representation is determined up to equivalence by the words formed by the labels of the edges in the cycles.
	These words are all primitive, and words corresponding to different components are not cyclic shifts of each other.
\end{corollary}

\begin{definition}\label{d:WD}
     Let $\abs{n} \geq 2$, and let us consider a one-dimensional strong affine representation of $\mathcal{P}_{\abs{n}}$ induced by a complete system of residues $D = \{ d_0,\dots,d_{\abs{n}-1} \}$ modulo $n$ such that  $d_0 < d_1 < \dots  < d_{\abs{n}-1}$.
     By Corollary~\ref{c:periodic}, the graph of the representation is determined up to isomorphism by a set of words $\WordsetD{D} \subseteq \mathcal{W}$ corresponding to the cycles of the graph.
     We say that $\WordsetD{D}$ is \emph{the set of words corresponding to the representation induced by $D$.}
     Let us recall that $\WordsetD{D}$ is a finite set of primitive words and it is closed under cyclic shifts.
\end{definition}

\begin{example}
    Figure~\ref{fig:example} shows a part of the graph of the representation of $\mathcal{P}_3$ corresponding to $n=3$ and $(d_0,d_1,d_2)=(1,6,14)$.
    The graph has $4$ connected components, containing cycles of lengths $3$, $1$, $2$
     and $1$. 
    The word associated with the periodic point $-5$ is $002$; similarly, the periodic points $-2$ and $-1$ are associated with the words $020$ and $200$, respectively. Observe that these words form a conjugacy class. We can read from the graph the words corresponding to the other cycles in a similar way, and then we obtain $\WordsetD{D} = \{ 002, 020, 200; 1; 12, 21; 2\}$.
    Although the figure shows only a part of the graph, the structure of the whole graph can be easily reconstructed using the fact that each vertex has $3$ outgoing edges, one for each label $0$, $1$ and $2$.
\end{example}
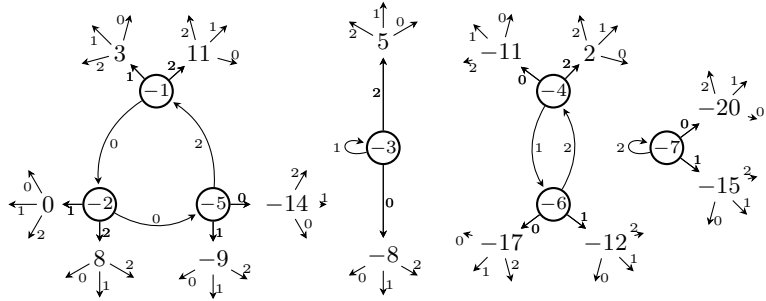
\begin{figure}[h]
\centering
\begin{tikzpicture}[scale=0.75,
        ->,>={Stealth[length=1mm, width=1mm]},shorten >=1pt,auto,node distance=0.3cm,
        node/.style={circle,inner sep=0.01cm,draw,thick}]
\node[node]                                 (A) at (-2,0)  {\footnotesize$-2$};
\node[node]                                 (B) at (-1,2)  {\footnotesize$-1$};
\node[node]                                 (C) at (0,0)   {\footnotesize$-5$};
\node[node]                                 (D) at (3,1)   {\footnotesize$-3$};
\node[node]                                 (E) at (6,0)   {\footnotesize$-6$};
\node[node]                                 (F) at (6,2)   {\footnotesize$-4$};
\node[node]                                 (G) at (8,1)   {\footnotesize$-7$};
\node[ellipse,inner sep=0.02cm,left=of A]          (d1)            {\scriptsize$0$} ;
\node[ellipse,inner sep=0.02cm,below= of A]        (d2)            {\scriptsize$8$};
\node[ellipse,inner sep=0.02cm,above left= of B]   (d3)            {\scriptsize$3$};
\node[ellipse,inner sep=0.02cm,above right= of B]  (d4)            {\scriptsize$11$};
\node[ellipse,inner sep=0.02cm,right=of C]         (d5)            {\scriptsize$-14$};
\node[ellipse,inner sep=0.02cm,below=of C]         (d6)            {\scriptsize$-9$};
\node[ellipse,inner sep=0.02cm,below=1cm of D]     (d7)            {\scriptsize$-8$} ;
\node[ellipse,inner sep=0.02cm,above=1cm of D]     (d8)            {\scriptsize$5$};
\node[ellipse,inner sep=0.02cm,above left= of F]   (d9)            {\scriptsize$-11$};
\node[ellipse,inner sep=0.02cm,above right= of F]  (d10)           {\scriptsize$2$};
\node[ellipse,inner sep=0.02cm,below left= of E]   (d11)           {\scriptsize$-17$};
\node[ellipse,inner sep=0.02cm,below right= of E]  (d12)           {\scriptsize$-12$};
\node[ellipse,inner sep=0.02cm,above right= of G]  (d13)           {\scriptsize$-20$};
\node[ellipse,inner sep=0.02cm,below right= of G]  (d14)           {\scriptsize$-15$} ;


%
\path   (A) edge[bend right] node[inner sep=0.02cm] {\tiny 0} (C)
        (C) edge[bend right] node[inner sep=0.02cm] {\tiny 2} (B)    
        (B) edge[bend right] node[inner sep=0.02cm] {\tiny 0} (A);
\path   (D) edge[loop left]  node[inner sep=0.05cm] {\tiny 1} (D);
\path   (E) edge[bend right] node[inner sep=0.02cm] {\tiny 2} (F)
        (F) edge[bend right] node[inner sep=0.02cm] {\tiny 1} (E);
\path   (G) edge[loop left]  node[inner sep=0.05cm] {\tiny 2} (G);
\draw (A)   -> node[inner sep=0.02cm] {\tiny 1} (d1);
\draw (A)   -> node[inner sep=0.02cm] {\tiny 2} (d2);
\draw (B)   -> node[inner sep=0.02cm] {\tiny 2} (d4);
\draw (B)   -> node[inner sep=0.02cm] {\tiny 1} (d3);
\draw (C)   -> node[inner sep=0.02cm] {\tiny 0} (d5);
\draw (C)   -> node[inner sep=0.02cm] {\tiny 1} (d6);
\draw (D)   -> node[inner sep=0.02cm] {\tiny 0} (d7);
\draw (D)   -> node[inner sep=0.02cm] {\tiny 2} (d8);
\draw (F)   -> node[inner sep=0.02cm] {\tiny 0} (d9);
\draw (F)   -> node[inner sep=0.02cm] {\tiny 2} (d10);
\draw (E)   -> node[inner sep=0.02cm] {\tiny 0} (d11);
\draw (E)   -> node[inner sep=0.02cm] {\tiny 1} (d12);
\draw (G)   -> node[inner sep=0.02cm] {\tiny 0} (d13);
\draw (G)   -> node[inner sep=0.02cm] {\tiny 1} (d14);
\draw (A)   -> node[inner sep=0.02cm] {\tiny 1} (d1);
\draw (A)   -> node[inner sep=0.02cm] {\tiny 2} (d2);
\draw (B)   -> node[inner sep=0.02cm] {\tiny 2} (d4);
\draw (B)   -> node[inner sep=0.02cm] {\tiny 1} (d3);
\draw (C)   -> node[inner sep=0.02cm] {\tiny 0} (d5);
\draw (C)   -> node[inner sep=0.02cm] {\tiny 1} (d6);
\draw (D)   -> node[inner sep=0.02cm] {\tiny 0} (d7);
\draw (D)   -> node[inner sep=0.02cm] {\tiny 2} (d8);
\draw (F)   -> node[inner sep=0.02cm] {\tiny 0} (d9);
\draw (F)   -> node[inner sep=0.02cm] {\tiny 2} (d10);
\draw (E)   -> node[inner sep=0.02cm] {\tiny 0} (d11);
\draw (E)   -> node[inner sep=0.02cm] {\tiny 1} (d12);
\draw (G)   -> node[inner sep=0.02cm] {\tiny 0} (d13);
\draw (G)   -> node[inner sep=0.02cm] {\tiny 1} (d14);
\draw (d1)  -> node[inner sep=0.02cm] {\tiny 0} +(120:0.75cm);
\draw (d1)  -> node[inner sep=0.02cm] {\tiny 1} +(180:0.75cm);
\draw (d1)  -> node[inner sep=0.02cm] {\tiny 2} +(240:0.75cm);
\draw (d2)  -> node[inner sep=0.02cm] {\tiny 0} +(210:0.75cm);
\draw (d2)  -> node[inner sep=0.02cm] {\tiny 1} +(270:0.75cm);
\draw (d2)  -> node[inner sep=0.02cm] {\tiny 2} +(330:0.75cm);
\draw (d3)  -> node[inner sep=0.02cm] {\tiny 0} +( 75:0.75cm);
\draw (d3)  -> node[inner sep=0.02cm] {\tiny 1} +(135:0.75cm);
\draw (d3)  -> node[inner sep=0.02cm] {\tiny 2} +(195:0.75cm);
\draw (d4)  -> node[inner sep=0.02cm] {\tiny 0} +(345:0.75cm);
\draw (d4)  -> node[inner sep=0.02cm] {\tiny 1} +( 45:0.75cm);
\draw (d4)  -> node[inner sep=0.02cm] {\tiny 2} +(105:0.75cm);
\draw (d5)  -> node[inner sep=0.02cm] {\tiny 0} +(300:0.75cm);
\draw (d5)  -> node[inner sep=0.02cm] {\tiny 1} +(  0:0.75cm);
\draw (d5)  -> node[inner sep=0.02cm] {\tiny 2} +( 60:0.75cm);
\draw (d6)  -> node[inner sep=0.02cm] {\tiny 0} +(210:0.75cm);
\draw (d6)  -> node[inner sep=0.02cm] {\tiny 1} +(270:0.75cm);
\draw (d6)  -> node[inner sep=0.02cm] {\tiny 2} +(330:0.75cm);
\draw (d7)  -> node[inner sep=0.02cm] {\tiny 0} +(210:0.75cm);
\draw (d7)  -> node[inner sep=0.02cm] {\tiny 1} +(270:0.75cm);
\draw (d7)  -> node[inner sep=0.02cm] {\tiny 2} +(330:0.75cm);
\draw (d8)  -> node[inner sep=0.02cm] {\tiny 0} +( 30:0.75cm);
\draw (d8)  -> node[inner sep=0.02cm] {\tiny 1} +( 90:0.75cm);
\draw (d8)  -> node[inner sep=0.02cm] {\tiny 2} +(150:0.75cm);
\draw (d9)  -> node[inner sep=0.02cm] {\tiny 0} +( 75:0.75cm);
\draw (d9)  -> node[inner sep=0.02cm] {\tiny 1} +(135:0.75cm);
\draw (d9)  -> node[inner sep=0.02cm] {\tiny 2} +(195:0.75cm);
\draw (d10) -> node[inner sep=0.02cm] {\tiny 0} +(345:0.75cm);
\draw (d10) -> node[inner sep=0.02cm] {\tiny 1} +( 45:0.75cm);
\draw (d10) -> node[inner sep=0.02cm] {\tiny 2} +(105:0.75cm);
\draw (d11) -> node[inner sep=0.02cm] {\tiny 0} +(165:0.75cm);
\draw (d11) -> node[inner sep=0.02cm] {\tiny 1} +(225:0.75cm);
\draw (d11) -> node[inner sep=0.02cm] {\tiny 2} +(285:0.75cm);
\draw (d12) -> node[inner sep=0.02cm] {\tiny 0} +(255:0.75cm);
\draw (d12) -> node[inner sep=0.02cm] {\tiny 1} +(315:0.75cm);
\draw (d12) -> node[inner sep=0.02cm] {\tiny 2} +( 15:0.75cm);
\draw (d13) -> node[inner sep=0.02cm] {\tiny 0} +(345:0.75cm);
\draw (d13) -> node[inner sep=0.02cm] {\tiny 1} +( 45:0.75cm);
\draw (d13) -> node[inner sep=0.02cm] {\tiny 2} +(105:0.75cm);
\draw (d14) -> node[inner sep=0.02cm] {\tiny 0} +(255:0.75cm);
\draw (d14) -> node[inner sep=0.02cm] {\tiny 1} +(315:0.75cm);
\draw (d14) -> node[inner sep=0.02cm] {\tiny 2} +( 15:0.75cm);
\end{tikzpicture} 
    \caption{The graph corresponding to $n=3$ and $(d_0,d_1,d_2)=(1,6,14)$}
    \label{fig:example}
\end{figure}

\subsection{Words of representations corresponding to arithmetic sequences}

If $D$ is an arithmetic sequence, then the words associated with periodic points can be described as repetends of $n$-ary expansions of certain rational numbers.
For $n=2$ this was proved in \cite{JL} and later for $n \geq 2$ in \cite{HW}.
In this subsection we recall these results and extend them for negative values of $n$.

\begin{theorem}\label{t:arithmetic B=A} 
	Let $\abs{n} \geq 2$ and let $D = \{ d_0,\dots,d_{\abs{n}-1} \}$ be a complete system of residues modulo $n$ such that $0\leq d_0\leq \abs{n-1}-1$.
	If $D$ is an arithmetic sequence, i.e., $d_i = d_0 + ih~(i \in \Sigma)$, where $h \in \N$ is relatively prime to $n$, then the periodic points of the corresponding representation are exactly the elements of $\mathcal{A}_n(D)$.
\end{theorem}

\begin{proof}
	This has been proved for positive $n$ in \cite[Theorem 3.1]{HW}.
	For negative $n$, the proof is almost the same: it suffices to prove that the restriction of $R$ to $\mathcal{A}_n(D)$ is injective (hence also bijective, by the pigeonhole principle). 
	Let $x,y,z\in\mathcal{A}_n(D)$ such that $x<y$ and $R(x)=R(y)=z$. 
	Then we have $x=nz+d_{i},y=nz+d_{j}$ for some $i,j \in \Sigma$, therefore $y-x=d_{j}-d_{i}=h(j-i) \geq h$. 
	Since the length of the interval defining $\mathcal{A}_n(D)$ in \ref{item:absorbing-negative} in Proposition~\ref{p:absorbing} is exactly $h$, the inequality $y-x \geq h$ can hold only if $x$ and $y$ are the left and the right endpoints of this interval (which must be integers then):
	\[
		x = \frac{-d_0n-d_{\abs{n}-1}}{n^2-1} \in \Z, \qquad y = \frac{-d_{\abs{n}-1}n-d_0}{n^2-1} \in \Z.
	\]
	Then we have $x \equiv d_{\abs{n}-1}~\pmod{n}$ and $y \equiv d_0~\pmod{n}$, which imply $R(x)=y$ and $R(y)=x$, contradicting the assumption $R(x)=R(y)$.
\end{proof}

\begin{definition}\label{d:value and reverse}
    For a word $c_{\ell-1}\ldots c_0=w\in\Sigma^+$ let $\overleftarrow{w}$ denote the reverse of $w$, i.e.,
    \(\overleftarrow{w}=c_0\ldots c_{\ell-1}\).
    Furthermore, let $\wval(w)$ denote the integer given by the $n$-ary expansion determined by $w$:
\[
    \wval(w)=(c_{\ell-1}\ldots c_0)_n=\sum_{i=0}^{\ell-1}c_{i}n^i.
\]
\end{definition}

The following lemma has been proved in \cite{HW} in the case when $n\geq2$. The proof is the same for $n\leq-2$, hence we omit it.
Theorem~\ref{t:arithmetic fractions} has also been obtained for positive $n$ in \cite{HW}, but we state it here in a slightly different form, and the proof is somewhat different if $n$ is negative.

\begin{lemma}[{\cite[Lemma 3.2]{HW}}]\label{l:arithmetic words}
	Let $\abs{n} \geq 2$ and let $D = \{ d_0,\dots,d_{\abs{n}-1} \}$ be a complete system of residues modulo $n$ such that $0\leq d_0\leq \abs{n-1}-1$.
	Assume that $D$ is an arithmetic sequence, i.e., $d_i = d_0 + ih~(i \in \Sigma)$, where $h \in \N$ is relatively prime to $n$.
	Let $w$ be a primitive word of length $\ell$ over $\Sigma$. 
	Then $w$ corresponds to a cycle if and only if the following number is an integer:
	\[
		-x:=\frac{d_{0}}{n-1}+\frac{h\wval\big(\overleftarrow{w}\big)}{n^{\ell}-1}.
	\]
    If this is the case, then $x$ is a periodic point, and $w$ is the word associated with $x$.
\end{lemma}

\begin{theorem}\label{t:arithmetic fractions}
	Let $\abs{n} \geq 2$ and let $D = \{ d_0,\dots,d_{\abs{n}-1} \}$ be a complete system of residues modulo $n$ such that $0\leq d_0\leq \abs{n-1}-1$. Let $r$ be the residue of $-d_0$ modulo $\abs{n-1}$, i.e., $r=\abs{n-1}-d_0$ if $d_0 > 0$ and $r=0$ if $d_0=0$. 
	Assume that $D$ is an arithmetic sequence, i.e., $d_i = d_0 + ih~(i \in \Sigma)$, where $h \in \N$ is relatively prime to $n$.
	  Then $\WordsetD{D}$ consists of the reverses of repetends of the $n$-ary expansions of the fractions in the following set:
	\begin{equation}\label{e:Q(r,h)}
        \left\{q\in I_{n} \ \middle|\  \exists a\in\Z\colon q=\frac{a}{h(n-1)} \text{, and } a\equiv r \pmod{n-1} \right\}.
	\end{equation}
    (Recall the definition of $I_n$ from Corollary~\ref{c:unit interval}, and note that each of the fractions $q \in I_n$ appearing in the above formula has a unique $n$-ary expansion except possibly when $q$ is one of the endpoints of $I_n$ (see Remark~\ref{r:repetend}). In the latter case we consider the expansion of $q$ that is of the form $(0.c_{-1}c_{-2}\dots)_n$.)
\end{theorem}

\begin{proof}
	Let $x$ be a periodic point and let $w=j_{0}\cdots j_{\ell-1}$ be a primitive word over $\Sigma$. 
	By Lemma~\ref{l:arithmetic words}, $w$ is the word associated with $x$ if and only if $-x=\frac{d_{0}}{n-1}+\frac{h\wval\big(\overleftarrow{w}\big)}{n^{\ell}-1}$, which is equivalent to
	\begin{align*}
		-\frac{x+\frac{d_{0}}{n-1}}{h} & =\frac{\wval\big(\overleftarrow{w}\big)/n^{\ell}}{1-1/n^{\ell}} = \wval\big(\overleftarrow{w}\big)\left(\frac{1}{n^{\ell}}+\frac{1}{n^{2\ell}}+\cdots\right) \\
		&  = j_{\ell-1}\frac{1}{n}+\cdots+j_{0}\frac{1}{n^{\ell}}+j_{\ell-1}\frac{1}{n^{\ell+1}}+\cdots+j_{0}\frac{1}{n^{2\ell}}+\cdots\\
		&  = (0. \, j_{\ell-1} \cdots j_{0} \, j_{\ell-1} \cdots j_{0} \cdots)_{n}.
	\end{align*}
	Thus, the word associated with $x$ is the reverse of the repetend of the $n$-ary expansion of $\frac{(n-1)(-x)-d_{0}}{h(n-1)}$. Let $$g\colon\R\to\R, g(x)=\frac{(n-1)(-x)-d_{0}}{h(n-1)}.$$ Observe that $g$ is an order reversing bijection from $\R$ onto $\R$.
    Note also that $g(x)$ is a fraction of the form as in \eqref{e:Q(r,h)} if and only if $x \in \Z$. 
	According to Theorem~\ref{t:arithmetic B=A}, to finish the proof it suffices to verify that 
    $x \in \mathcal{A}_n(D)$ if and only if $g(x) \in I_n$.
    We will show this seperately for $n\geq2$ and $n\leq-2$.
    
    If $n$ is positive, then we have $d_{n-1}=d_0+(n-1)h$. Using Proposition~\ref{p:absorbing} and Corollary~\ref{c:unit interval}, we need to prove that
    \begin{equation}\label{e:A_n positive}
    \frac{-(d_0+(n-1)h)}{n-1}\leq x\leq\frac{-d_0}{n-1}
    \end{equation}
    if and only if
    \[
    0\leq g(x)\leq1.
    \]
    Since $g$ is an order reversing bijection, by applying $g$ on the inequalities in \eqref{e:A_n positive} we get an equivalent inequality:
    \[
    1=g\left(\frac{-(d_0+(n-1)h)}{n-1}\right)\geq g(x)\geq g\left(\frac{-d_0}{n-1}\right)=0.
    \]
      
    If $n$ is negative, then we have $d_{\abs{n}-1}=d_0+(-n-1)h$. Using Proposition~\ref{p:absorbing} and Corollary~\ref{c:unit interval}, we need to show that for any $x\in\Z$, we have
    \begin{equation}\label{e:A_n negative}
    \frac{-d_0n-(d_0+(-n-1)h)}{n^2-1}\leq x\leq\frac{-(d_0+(-n-1)h)n-d_0}{n^2-1}
    \end{equation}
    if and only if
    \begin{equation}\label{e:I_n negative}
        \frac{-n}{n-1}\leq g(x)\leq\frac{-1}{n-1}.
    \end{equation}
    We can simplify the expressions in \eqref{e:A_n negative}:
    \[
    \frac{-d_0n-(d_0+(-n-1)h)}{n^2-1}=\frac{-d_0(n+1)+(n+1)h}{(n+1)(n-1)}=\frac{-d_0+h}{n-1},
    \]
    and    
    \[
    \frac{-(d_0+(-n-1)h)n-d_0}{n^2-1}=\frac{-d_0(n+1)+n(n+1)h}{(n+1)(n-1)}=\frac{-d_0+nh}{n-1}.
    \]
    Since $g$ is an order reversing bijection and 
    \[
    g\left(\frac{-d_0+h}{n-1}\right)=\frac{-1}{n-1}\quad\text{and}\quad g\left(\frac{-d_0+nh}{n-1}\right)=\frac{-n}{n-1},
    \]
    we see that \eqref{e:A_n negative} is indeed equivalent to \eqref{e:I_n negative}.
\end{proof}

Next we introduce some notation that allows us to formulate Theorem~\ref{t:arithmetic fractions} in a more compact form.

\begin{definition}\label{d:qw}
    For a word $w=c_0c_1\dots c_{\ell-1}$, let $\fraction(w)$ denote the following fraction:
    \[
    \fraction(w)=(0.c_{\ell-1}\dots c_0 c_{\ell-1}\dots c_0\dots)_n=\frac{\wval\big(\overleftarrow{w}\big)}{n^\ell-1}=\frac{\wval\big(\overleftarrow{w}\big)}{\frac{n^\ell-1}{n-1}(n-1)}.
    \]
\end{definition}
Clearly, the repetend of $\fraction(w)$ is the reverse of $w$ and we have $\fraction(w) \in I_n$.
\begin{definition}\label{d:QrhWrh}
    If $r\in\{0,1,\dots,\abs{n-1}-1\}$ and $h\in\N$ is relatively prime to $n$, then let
    \[
    \rQuoth{r}{h}\coloneqq\left\{q\in I_{n} \mathrel{}\middle|\mathrel{} \exists a\in\Z\colon q=\frac{a}{h(n-1)} \text{, and } a\equiv r \pmod{n-1}  \right\},
    \]
    and let 
    \[
    \rWordh{r}{h}\coloneqq\left\{w\in \mathcal{W} : \fraction(w)\in\rQuoth{r}{h}\right\}.
    \]
\end{definition}

\begin{corollary}\label{c:WD=rWH}
    Let $\abs{n} \geq 2$ and let $D = \{ d_0,\dots,d_{\abs{n}-1} \}$ be a complete system of residues modulo $n$ such that $0\leq d_0\leq \abs{n-1}-1$. Let $r$ be the residue of $-d_0$ modulo $\abs{n-1}$, i.e., $r=\abs{n-1}-d_0$ if $d_0 > 0$ and $r=0$ if $d_0=0$. 
	Assume that $D$ is an arithmetic sequence, i.e., $d_i = d_0 + ih~(i \in \Sigma)$, where $h \in \N$ is relatively prime to $n$. Then we have $\WordsetD{D}=\rWordh{r}{h}$. 
\end{corollary}

\begin{remark}\label{r:rWh cyclic shifts}
    Recall from Definition~\ref{d:WD} that $\WordsetD{D}$ is always closed under cyclic shifts, hence the same is true for $\rWordh{r}{h}$. The latter fact can also be verified directly by considering the repetends of the fractions in $\rQuoth{r}{h}$.
\end{remark}

\begin{example}\label{ex:Q(0,n^l-1)}
    Let us consider the case $r=0$, $h=n^\ell-1$ with $n \geq 2$.
    In this case we have    
    \[
    \rQuoth{r}{h}\coloneqq\left\{\frac{0}{n^\ell-1},\  \frac{1}{n^\ell-1},\  \frac{2}{n^\ell-1},\  \dots,\  \frac{n^\ell-1}{n^\ell-1} \right\}.
    \]
    The repetend of a fraction $a/(n^\ell-1)$ with $0 \leq a \leq n^\ell-1$ is the unique primitive word $w$ such that $\wval(w^k)=a$ for some positive integer $k$.
    Here $k$ is a divisor of $\ell$ and $\abs{w} = \ell/k$.
    Conversely, if $\abs{w}$ divides $\ell$ and $a = \wval(w^k)$ with $k = \ell/\abs{w}$, then $a/(n^\ell-1) \in \rQuoth{r}{h}$.
    We see that the repetends of the fractions in $\rQuoth{r}{h}$ are exactly the primitive words whose length is a divisor of $\ell$; therefore, $\rWordh{r}{h}$ consists of these words, too (note that we have to take the reverses of the repetends, but this does not change the set that we obtain).
    This result appears as Corollary 5.8 in \cite{JL} for $n=2$.

\end{example}

\section{The closure operator on primitive words}\label{sec:arithmetic}
We have seen in the previous section that if $D$ is an arithmetic sequence, then $\WordsetD{D}$ can be calculated using $n$-ary expansions of certain fractions.
In this section we will use this to construct a closure system on $\mathcal{W}$ that will help us characterize the sets $\WordsetD{D}$ corresponding to arithmetic sequences. By Corollary~\ref{c:WD=rWH}, these sets coincide with the sets of the form $\rWordh{r}{h}$.

\subsection{The lattice of closed sets}

Our goal in this section is to prove that the sets $\rWordh{r}{h}$ together with the empty set form a lattice under inclusion.

\begin{proposition}\label{p:leq}
    Let $h_1,h_2\in\N$ such that $\gcd(h_1,n)=\gcd(h_2,n)=1$, and let $r_1,r_2\in\{0,1,\dots,\abs{n-1}-1\}$.
    Then we have $\rQuoth{r_1}{h_1}\subseteq\rQuoth{r_2}{h_2}$, or, equivalently, $\rWordh{r_1}{h_1}\subseteq\rWordh{r_2}{h_2}$, if and only if
    \[
    h_1\mid h_2 \quad\text{and}\quad r_1\cdot\frac{h_2}{h_1}\equiv r_2\pmod{n-1}.
    \]
\end{proposition}
\begin{proof}
    Assume that $\rQuoth{r_1}{h_1}\subseteq\rQuoth{r_2}{h_2}$.   
    If $h_1=1$, then there is $b\in\Z$ such that 
    \[\frac{(n-1)b+r_1}{n-1}\in\rQuoth{r_1}{1},\] 
    because the length of $I_n$ is 1. It is clear that $h_1 \mid h_2$, as $h_1=1$. Since
    \[
    \frac{(n-1)b+r_1}{n-1}=\frac{h_2(n-1)b+h_2\cdot r_1}{h_2(n-1)}\in\rQuoth{r_2}{h_2},
    \] we have $h_2\cdot r_1=r_1\frac{h_2}{h_1}\equiv r_2\pmod{n-1}$.
    
    If $h_1>1$, then $\abs{\rQuoth{r_1}{h_1}}\geq2$, namely $\frac{(n-1)b+r_1}{h_1(n-1)}\in\rQuoth{r_1}{h_1}$ and $\frac{(n-1)(b+1)+r_1}{h_1(n-1)}\in\rQuoth{r_1}{h_1}$ for some $b \in \Z$, because $I_n$ is of length 1.
    Since $\rQuoth{r_1}{h_1}\subseteq\rQuoth{r_2}{h_2}$, we have
    \begin{align*}
        \frac{(n-1)b+r_1}{h_1(n-1)} &= \frac{((n-1)b+r_1)\cdot\frac{h_2}{h_1}}{h_2(n-1)} \in \rQuoth{r_2}{h_2},
        \;\text{ and }\;\\[1mm]
        \frac{(n-1)(b+1)+r_1}{h_1(n-1)} &= \frac{((n-1)(b+1)+r_1)\cdot\frac{h_2}{h_1}}{h_2(n-1)} \in \rQuoth{r_2}{h_2}.
    \end{align*}
    This means that the numerators of the above fractions are integers, and 
    \begin{equation}\label{eq: r2}
        ((n-1)b+r_1)\cdot\frac{h_2}{h_1}\equiv((n-1)(b+1)+r_1)\cdot\frac{h_2}{h_1}\pmod{n-1},
    \end{equation}
    as both numbers are congruent to $r_2$ modulo $n-1$ by the definition of $\rQuoth{r_2}{h_2}$.
    Thus the difference of the left and right hand sides of \eqref{eq: r2} is divisible by $n-1$, i.e., $(n-1)\mid(n-1)\frac{h_2}{h_1}$ and so $h_1\mid h_2$. Moreover, since $\frac{h_2}{h_1}\in\Z$, we have
\[
((n-1)b+r_1)\cdot\frac{h_2}{h_1}\equiv
(n-1)b\cdot\frac{h_2}{h_1}+r_1\cdot\frac{h_2}{h_1}\equiv
r_1\cdot\frac{h_2}{h_1}\equiv r_2\pmod{n-1}.
\]

Conversely, if $h_1\mid h_2$ and $r_1\cdot\frac{h_2}{h_1}\equiv r_2\pmod{n-1}$, then $r_1\cdot\frac{h_2}{h_1}=c(n-1)+r_2$ for some integer $c$. Therefore, for any $\frac{(n-1)b+r_1}{h_1(n-1)}\in\rQuoth{r_1}{h_1}$, we have 
\[\frac{(n-1)b+r_1}{h_1(n-1)}=\frac{\frac{h_2}{h_1}(n-1)b+c(n-1)+r_2}{h_2(n-1)},\]
which means that $\frac{(n-1)b+r_1}{h_1(n-1)}\in\rQuoth{r_2}{h_2}$. Thus we can conclude that $\rQuoth{r_1}{h_1}\subseteq\rQuoth{r_2}{h_2}$.
\end{proof}

\begin{definition}\label{def:latn}
Let
\[
\latn\coloneqq\Big(\{0,1,2,\dots,\abs{n-1}-1\}\times\{h\in\N\colon\gcd(h, n)=1\}\Big) \,\cup\, \{\bot\},
\]
 and let $\leq$ be the partial order on $\latn$ such that $\bot$ is the least element and
\[ 
(r_1,h_1)\leq(r_2,h_2)\quad\ \  \iff \quad\ \  h_1\mid h_2 \quad \text{and}\quad r_1\cdot\frac{h_2}{h_1}\equiv r_2\pmod{n-1}.
\]
\end{definition}
By Proposition \ref{p:leq}, $\leq$ is indeed a partial order; moreover, the poset $(\latn\setminus\{\bot\};\leq)$ is isomorphic to the poset of the sets $\rQuoth{r}{h}$, ordered by inclusion.

\begin{theorem}\label{t:lat}
    The relation $\leq$ is a lattice order on $\latn$. Furthermore, $(r_1,h_1)\wedge(r_2,h_2)\neq \bot$ if and only if 
    \begin{equation}\label{cond}
        r_1\frac{h_2}{\gcd(h_1,h_2)}\equiv r_2\frac{h_1}{\gcd(h_1,h_2)}\pmod{n-1}.
    \end{equation}
\end{theorem}  

\begin{proof}
First we show that any finite set of elements has a least common upper bound. (Of course, it would suffice to consider just two elements, but later we will need the formula for the join of finitely many elements.)

We only need to consider the case when none of the elements are the bottom element $\bot$. Let us consider the elements $(r_1,h_1),\dots,(r_k,h_k)$, and let $h'=\lcm(h_1,\dots,h_k)$. Then for every $i\in [k]$, we have $(r_i,h_i)\leq(r_i',h')$, where $r_i'\equiv r_i\frac{h'}{h_i}\pmod{n-1}$.
(Here, and in the sequel, $[k]$ denotes the set $\{1,\dots,k\}$.)
First we find a common upper bound $(r,h)$ for the elements $(r'_i,h')$; it is clear that $(r,h)$ is then also a common upper bound for the elements $(r_i,h_i)$. Thus we need to find $m\in\N$ and $r\in\{0,\dots,\abs{n-1}-1\}$ such that
\[
r_1'm\equiv\dots\equiv r_k'm\equiv r\pmod{n-1},
\]
in which case the common upper bound is $(r,h'm)$. 
From the above congruences it follows that $r_i'm\equiv r_j'm\pmod{n-1}$; therefore,  $m\equiv0\pmod{\frac{n-1}{\gcd(n-1,r_i'-r_j')}}$. This holds for every $i,j\in[k]$, so 

\begin{equation}\label{e:m'}
    m' := \lcm\left\{\frac{n-1}{\gcd(n-1,r_i'-r_j')}\mathrel{}\middle|\mathrel{}i,j\in[k]
\right\}
\text{ divides } m.
\end{equation}
Conversely, if $m' \mid m$, then for every $i,j\in[k]$, we have $r_i'm\equiv r_j'm\pmod{n-1}$. Thus, if $r\equiv r_1'm\pmod{n-1}$, then $(r,h' m)$ is a common upper bound of the elements $(r'_i,h')$.

Let us choose $m=m'$, $r\equiv r_1'm'\pmod{n-1}$ and $h=h'm'$.
From the above considerations it follows that $(r,h)$ is then a common upper bound of the elements $(r_i,h_i)$; we will prove that it is actually the least common upper bound.
Let us take an arbitrary common upper bound $(r^*,h^*)$; we need to show that $(r,h)\leq(r^*,h^*)$. 
We know that $h_1,\ldots,h_k\mid h^*$, so $h'\mid h^*$. Then for every $i\in[k]$, we have $r_i'\frac{h^*}{h'}\equiv r_i\frac{h'}{h_i}\frac{h^*}{h'}\equiv r_i\frac{h^*}{h_i} \equiv r^*\pmod{n-1}$, where the last congruence holds because $(r_i,h_i) \leq (r^*,h^*)$. 
Thus $r_i'\frac{h^*}{h'}\equiv r_j'\frac{h^*}{h'}\pmod{n-1}$ for all $i,j \in [k]$, and this together with \eqref{e:m'} implies $m'\mid \frac{h^*}{h'}$. This means that 
\[
h=h'm' \mathrel{\bigg|} h'\frac{h^*}{h'}=h^* 
\quad \text{and}\quad
r\frac{h^*}{h}=r_1'm'\frac{h^*}{h'm'}=r_1'\frac{h^*}{h'}
\equiv r^*\pmod{n-1},
\]
proving $(r,h)\leq(r^*,h^*)$.

Next we prove that any two elements $(r_1,h_1)$ and $(r_2,h_2)$ have a greatest common lower bound.
If the only common lower bound of two elements is $\bot$ then it is the greatest common lower bound.

Otherwise there is a common lower bound of the form $(r,h)$. Then $h$ is a common divisor of $h_1$ and $h_2$, i.e., $h \mid h'$, where $h'=\gcd(h_1,h_2)$.
Let $r'\in\{0,\dots,\abs{n-1}-1\}$ such that $r'\equiv r\frac{h'}{h}\pmod{n-1}$; then we have $(r,h)\leq(r',h')$. 
The element $(r',h')$ is also a common lower bound of $(r_1,h_1)$ and $(r_2,h_2)$, since for every $i\in\{1,2\}$, we have $h'\mid h_i$ and 
\[
r'\frac{h_i}{h'}\equiv r\frac{h'}{h}\frac{h_i}{h'}\equiv r\frac{h_i}{h}\equiv  r_i \pmod{n-1},
\]
where the last congruence follows from $(r,h) \leq (r_i,h_i)$. 
Therefore, every maximal common lower bound has the form $(p,h')$.
(Note that maximal common lower bounds do exist, as there are only finitely many elements below any given element of $\latn$.)
If $(p,h')$ is a common lower bound, then
$x=p$ is a solution to the system of congruences
\begin{equation}\label{sys}
    \left\{\begin{array}{l}
     x\frac{h_1}{h'}\equiv r_1\pmod{n-1}  \\[1mm]
     x\frac{h_2}{h'}\equiv r_2\pmod{n-1}.  
\end{array}\right. 
\end{equation} Since $\gcd\left(\frac{h_1}{h'},\frac{h_2}{h'}\right)=1$, the Diophantine equation $s\frac{h_1}{h'}+t\frac{h_2}{h'}=1$ is solvable. If $(s_0,t_0)$ is a solution, then the set of all solutions is $\{(s_\alpha,t_\alpha)\colon\alpha\in\Z\}$,
where $s_\alpha=s_0+\frac{h_2}{h'}\alpha$ and $t_\alpha=t_0-\frac{h_1}{h'}\alpha$.

Let us multiply the first congruence of \eqref{sys} by $s_0$ and the second congruence by $t_0$, and let us take the sum of the two congruences:
\begin{equation}\label{bize}
s_0r_1+t_0r_2 \equiv s_0x\frac{h_1}{h'} + t_0x\frac{h_2}{h'}\equiv x\left(s_0\frac{h_1}{h'}+t_0\frac{h_2}{h'}\right) \equiv x \pmod{n-1}.
\end{equation}
This shows that the least nonnegative solution $p$ of \eqref{sys} satisfies $p \equiv s_0r_1+t_0r_2 \pmod{n-1}$; moreover, $p$ is the unique solution in $\{0,1,\dots,\abs{n-1}-1 \}$.
Thus there is a unique common lower bound of the form $(p,h')$, and this is necessarily the greatest common lower bound.

Finally we prove that $(r_1,h_1) \wedge (r_2,h_2)\neq\bot$ if and only if \eqref{cond} holds.
Assume first that $(r_1,h_1) \wedge (r_2,h_2) = (p,h')$, where $h' = \gcd(h_1,h_2)$. Then $p$ is a solution of \eqref{sys}, hence we have $p \equiv s_0r_1+t_0r_2 \pmod{n-1}$ by \eqref{bize}. A similar computation shows that $p \equiv s_\alpha r_1+t_\alpha r_2 \pmod{n-1}$ for any integer $\alpha$.
Subtracting the latter two congruences, we get
\[
    0 \equiv s_\alpha r_1+t_\alpha r_2-(s_0r_1+t_0r_2)  \equiv \alpha r_1\frac{h_2}{h'}-\alpha r_2\frac{h_1}{h'} \pmod{n-1}.
\] 
Here $\alpha=1$ gives \eqref{cond}.

Conversely assume \eqref{cond}, and let us prove that there is a $p\in\{0,\dots,\abs{n-1}-1\}$ such that \eqref{sys} holds for $x=p$.
First we verify that both congruences are solvable, and then we show that the two congruences have a common solution.

The first congruence has a solution if and only if $d_1\coloneqq\gcd(\frac{h_1}{h'},n-1)\mid r_1$. Since $d_1\mid r_2\frac{h_1}{h'}$, we can deduce from \eqref{cond} that $d_1\mid r_1\frac{h_2}{h'}$. This implies $d_1 \mid r_1$, as $\gcd(d_1,\frac{h_2}{h'})=1$, hence the first congruence has a solution. The solvability of the second congruence can be proved in a similar way.

Let $d_1=\gcd(\frac{h_1}{h'},n-1)$ and $d_2=\gcd(\frac{h_2}{h'},n-1)$, then \eqref{sys} is equivalent to
\[
    \left\{\begin{array}{l}
     x\equiv q_1\pmod{\frac{n-1}{d_1}}  \\[1mm]
     x\equiv q_2\pmod{\frac{n-1}{d_2}},  
    \end{array}\right. 
\]
for suitable integers $q_1,q_2$. Since $\gcd(d_1, d_2)=1$, we have $\gcd\left(\frac{n-1}{d_1},\frac{n-1}{d_2}\right)=\frac{n-1}{d_1d_2}$, thus the system above is solvable if and only if $\frac{n-1}{d_1d_2}\mid q_2-q_1$.

Substituting \eqref{sys} into \eqref{cond}, we can deduce that 
\[
    q_1\frac{h_1}{h'}\frac{h_2}{h'}\equiv q_2\frac{h_2}{h'}\frac{h_1}{h'}\pmod{n-1}.
\]
Observe that $\gcd(\frac{h_1h_2}{h'h'},n-1) = d_1d_2$, because $\frac{h_1}{h'}$ and $\frac{h_2}{h'}$ are relatively prime.
Thus the congruence above is equivalent to $q_1\equiv q_2\pmod{\frac{n-1}{d_1d_2}}$, which implies $\frac{n-1}{d_1d_2} \mathrel{\big|} q_1-q_2$, as required.
\end{proof}

\subsection{The closure operator}

We will see that the sets of words corresponding to one-dimensional strong affine representations of $\mathcal{P}_{\abs{n}}$ induced by arithmetic sequences (together with $\emptyset$ and $\mathcal{W}$) form a closure system (see Theorem~\ref{t:closure}), and we will give a method to compute the closure $\cl(W)$ of a set of primitive words $W$.
This provides a characterization of the aforementioned representations of the polycyclic monoids.
First we need to introduce some notation in order to state these results.  

Recall the notations $\fraction(w)$ and $\rQuoth{r}{h}$ from Definitions~\ref{d:qw} and \ref{d:QrhWrh}.
Let $q=\frac{a}{h(n-1)}\in\rQuoth{r}{h}$; then $q=\fraction(w)$ if and only if $\frac{a}{h}$ is an expansion of the reduced form of $\fraction(w) \cdot (n-1) = \frac{\wval\big(\overleftarrow{w}\big)}{\frac{n^\ell-1}{n-1}}$. Let $\denom(w)$ and $\numer(w)$ denote the denominator and numerator of this reduced fraction:
\begin{equation}\label{e:def h es a}    
\denom(w)=\frac{\frac{n^\ell-1}{n-1}}{\gcd\left(\frac{n^\ell-1}{n-1},\wval\big(\overleftarrow{w}\big)\right)}\quad\text{and}\quad  \numer(w)=\frac{\wval\big(\overleftarrow{w}\big)}{\gcd\left(\frac{n^\ell-1}{n-1},\wval\big(\overleftarrow{w}\big)\right)}.
\end{equation}
Note that the arguments of the greatest common divisors above may be negative. In this case we choose the sign of the gcd in the first fraction in such a way that $\denom(w)$ becomes positive. We use the same sign in the second fraction, allowing $\numer(w)$ to be negative.
The least set $\rQuoth{r}{h}$ containing $\fraction(w)=\frac{\numer(w)}{\denom(w)(n-1)}$ is $\rQuoth{\rem(w)}{\denom(w)}$, where 
\begin{equation}\label{e:def r}
    \rem(w)\in\{0,\dots,\abs {n-1}-1\} \text{ such that } \numer(w)\equiv\rem(w)\pmod{n-1}.
\end{equation}

Now we are ready to state our main result that characterizes the representations in question.

\begin{theorem}\label{t:closure}
    For any integer $n$ with $\abs{n} \geq 2$ and for any nonempty finite set $W \subseteq \mathcal{W}$ of primitive words over $\Sigma$, there exists a least set $\cl(W)$ of words such that
    \begin{enumerate}[(i)]
        \item $\cl(W) \supseteq W$, and
        \item $\cl(W) = \WordsetD{D}$, where $D$ is an arithmetic sequence.
    \end{enumerate}
    The closure $\cl(W)$ can be computed using the join operation of the lattice $\latn$ as follows: 
    \[
    \cl(W)=\rWordh{r}{h}, \text{ where }
    (r,h)=\bigvee_{w\in W}(\rem(w),\denom(w)).
    \]
\end{theorem}
\begin{proof}
    According to Corollary~\ref{c:WD=rWH}, we need to prove that there is a least set of the form $\rWordh{r}{h}$ that contains $W$.
    By Proposition~\ref{p:leq}, the ordering of the sets $\rWordh{r}{h}$ is isomorphic to the ordering of $\latn$, and Theorem~\ref{t:lat} shows that any finite subset of $\latn$ has a least upper bound.
\end{proof}

\begin{remark}\label{rem:closureoperator}
    Let us consider the following family of sets of words:
    \[
    \{ \WordsetD{D} : D \text{ is an arithmetic sequence}\} \cup \{ \emptyset, \mathcal{W} \}.
    \]
    This is a closure system on $\mathcal{W}$, and the set $\cl(W)$ appearing in the corollary above is the closure of $W$ whenever $W$ is a nonempty finite set of primitive words.
    (If $W \subseteq \mathcal{W}$ is an infinite set of words, then the closure of $W$ is $\mathcal{W}$, but this is not relevant for one-dimensional strong affine representations of $\mathcal{P}_{\abs{n}}$.)
    Observe that according to \eqref{e:def h es a} and \eqref{e:def r}, we have $\cl(\{w\}) = \rWordh{\rem(w)}{\denom(w)}$.
\end{remark}

We summarize below how one can compute the closure of a given finite set of primitive words $W=\{w_1,w_2,\dots,w_k\}$ along the lines of the proof of Theorem~\ref{t:lat}, and then we also present some concrete examples for this computation.
\begin{description}\addtolength{\itemsep}{2mm}
    \item[Step 1.] Compute the fractions $\fraction(w_i)$ by Definition~\ref{d:qw}.
    \item[Step 2.] Reduce these fractions and use \eqref{e:def h es a} and \eqref{e:def r} to obtain the numbers $h_i=\denom(w_i)$, $a_i=\numer(w_i)$ and $r_i=\rem(w_i)$.
    \item[Step 2.] Compute $h'=\lcm(h_1,h_2,\dots,h_k)$.
    \item[Step 3.] Compute $r_i'\equiv r_i\frac{h'}{h_i}\pmod{n-1}$.
    \item[Step 4.] Compute $m'$ from \eqref{e:m'}.
    \item[Step 5.] Compute $r\equiv r_1'm'\pmod{n-1}$ and $h=h'm'$; this gives the join $(r,h) = (r_1,h_1) \vee (r_2,h_2) \vee \dots \vee (r_k,h_k)$. 
    \item[Step 6.] Compute $\rQuoth{r}{h}$ by Definition~\ref{d:QrhWrh}.
    \item[Step 7.] Finally, $\cl(W) = \rWordh{r}{h}$ consists of the reverses of the repetends of the $n$-ary expansions of the fractions in $\rQuoth{r}{h}$.
    \end{description}


\subsection{Examples for computing the closure}

We present some concrete numerical examples for computing the closure of a finite set of primitive words using the method outlined above.

\begin{example}
    Let us find the closure of the set $W_1 = \{42,312400\}$ for $n=5$.
    First we compute the fractions $\fraction(w_i)$:
    \begin{align*}
        \wval\big(\overleftarrow{w_1}\big) &= (24)_5=14,\\[2mm]
        \wval\big(\overleftarrow{w_2}\big) &= (004213)_5=558,\\[2mm]
        \fraction(w_1) &= \frac{14}{5^2-1} = \frac{14}{6 \cdot 4} = \frac{7}{3 \cdot 4},\\[2mm]
        \fraction(w_2) &= \frac{558}{5^6-1} = \frac{558}{3906 \cdot 4}  = \frac{1}{7 \cdot 4}.
    \end{align*}
    The corresponding numbers $\denom(w)$, $\numer(w)$ and $\rem(w)$ are the following:
    \[
    \begin{array}{lll}
        h_1:=\denom(w_1)=3, &\numer(w_1)=7, &r_1:=\rem(w_1)=3; \\ 
        h_2:=\denom(w_2)=7, &\numer(w_2)=1,  &r_2:=\rem(w_2)=1.
    \end{array}
    \]
    We have $h'=\lcm(h_1,h_2)=21$ and
    \begin{align*}
    r'_1 &\equiv r_1 \cdot \frac{h'}{h_1} \equiv 3 \cdot \frac{21}{3} \equiv 1\pmod{4},\\
    r'_2 &\equiv r_2 \cdot \frac{h'}{h_2} \equiv 1 \cdot \frac{21}{7} \equiv 3\pmod{4},\\
    m' &=\frac{n-1}{\gcd(n-1,r'_2-r'_1)}=\frac{4}{\gcd(4,2)}=2,\\
    r &\equiv r'_1 \cdot m' \equiv 1 \cdot 2 \equiv 2\pmod{4},\\
    h &= h' \cdot m' = 21 \cdot 2 = 42.
    \end{align*}
Therefore, the closure of $W_1$ is determined by the join $(r_1,h_1) \vee (r_2,h_2) = (r,h) = (2,42)$:
    \begin{align*}
        \rQuoth{2}{42} &= \left\{ \frac{a}{42 \cdot 4} \mathrel{}\middle|\mathrel{} 0 \leq a \leq 168 \text{ and } a \equiv 2\pmod{4} \right\}\\ 
        &= \left\{ \frac{2}{168}, \frac{6}{168}, \dots, \frac{166}{168} \right\}  = \left\{ \frac{1}{84}, \frac{3}{84}, \dots, \frac{83}{84} \right\}.  
    \end{align*}    
The closure $\cl(W_1)$ consists of the reverses of the repetends of the $5$-ary expansions of these fractions:
{
\everymath={\displaystyle}
\[
\begin{array}{r@{\ }lcr@{\ }lcr@{\ }l}
\frac{1}{84}  &= (0.001221\dots)_5 & \quad & \frac{29}{84} &= (0.133034\dots)_5 & \quad & \frac{57}{84} &= (0.314402\dots)_5 \\[4mm]
\frac{3}{84}  &= (0.004213\dots)_5 & \quad & \frac{31}{84} &= (0.141031\dots)_5 & \quad & \frac{59}{84} &= (0.322344\dots)_5 \\[4mm]
\frac{5}{84}  &= (0.012210\dots)_5 & \quad & \frac{33}{84} &= (0.144023\dots)_5 & \quad & \frac{61}{84} &= (0.330341\dots)_5 \\[4mm]
\frac{7}{84}  &= (0.02    \dots)_5 & \quad & \frac{35}{84} &= (0.20    \dots)_5 & \quad & \frac{63}{84} &= (0.3     \dots)_5 \\[4mm]
\frac{9}{84}  &= (0.023144\dots)_5 & \quad & \frac{37}{84} &= (0.210012\dots)_5 & \quad & \frac{65}{84} &= (0.341330\dots)_5 \\[4mm]
\frac{11}{84} &= (0.031141\dots)_5 & \quad & \frac{39}{84} &= (0.213004\dots)_5 & \quad & \frac{67}{84} &= (0.344322\dots)_5 \\[4mm]
\frac{13}{84} &= (0.034133\dots)_5 & \quad & \frac{41}{84} &= (0.221001\dots)_5 & \quad & \frac{69}{84} &= (0.402314\dots)_5 \\[4mm]
\frac{15}{84} &= (0.042130\dots)_5 & \quad & \frac{43}{84} &= (0.223443\dots)_5 & \quad & \frac{71}{84} &= (0.410311\dots)_5 \\[4mm]
\frac{17}{84} &= (0.100122\dots)_5 & \quad & \frac{45}{84} &= (0.231440\dots)_5 & \quad & \frac{73}{84} &= (0.413303\dots)_5 \\[4mm]
\frac{19}{84} &= (0.103114\dots)_5 & \quad & \frac{47}{84} &= (0.234432\dots)_5 & \quad & \frac{75}{84} &= (0.421300\dots)_5 \\[4mm]
\frac{21}{84} &= (0.1     \dots)_5 & \quad & \frac{49}{84} &= (0.24    \dots)_5 & \quad & \frac{77}{84} &= (0.42    \dots)_5 \\[4mm]
\frac{23}{84} &= (0.114103\dots)_5 & \quad & \frac{51}{84} &= (0.300421\dots)_5 & \quad & \frac{79}{84} &= (0.432234\dots)_5 \\[4mm]
\frac{25}{84} &= (0.122100\dots)_5 & \quad & \frac{53}{84} &= (0.303413\dots)_5 & \quad & \frac{81}{84} &= (0.440231\dots)_5 \\[4mm]
\frac{27}{84} &= (0.130042\dots)_5 & \quad & \frac{55}{84} &= (0.311410\dots)_5 & \quad & \frac{83}{84} &= (0.443223\dots)_5
\end{array}
\]
}
Thus $\cl(W_1)$ contains $42$ words; to obtain a better overview of these words, we list below them only up to cyclic shifts (recall from Remark~\ref{r:rWh cyclic shifts} that if a word appears in a representation of $\mathcal{P}_5$, then all of its cyclic shifts also appear):
\[
1,\; 3,\; 02,\; 24,\; 001221,\; 003124,\; 044132,\; 014113,\;  033143,\; 223443.
\]
\end{example}

\begin{example}
    Let us now find the closure of the set $W_2 = \{42,312400\}$ for $n=-5$.
    (This is the same set $W_1$ as in the previous example, but we use a different notation to indicate that we consider $n=-5$ here, thus the closure operator is different.)
    First we compute the fractions $\fraction(w_i)$:
    \begin{align*}      
        \wval\big(\overleftarrow{w_1}\big) &= (24)_{-5}=-6,\\[2mm]
        \wval\big(\overleftarrow{w_2}\big) &= (004213)_{-5}=-452,\\[2mm]
        \fraction(w_1) &= \frac{-6}{(-5)^2-1} = \frac{-6}{-4\cdot (-6)} = \frac{3}{2 \cdot (-6)},\\[2mm]
        \fraction(w_2) &= \frac{-452}{(-5)^6-1} = \frac{-452}{-2604 \cdot (-6)}  = \frac{113}{651 \cdot (-6)}.
    \end{align*}
    The corresponding numbers $\denom(w)$, $\numer(w)$ and $\rem(w)$ are the following:
    \[
    \begin{array}{lll}
        h_1:=\denom(w_1)=2, &\numer(w_1)=3, &r_1:=\rem(w_1)=3; \\ 
        h_2:=\denom(w_2)=651, &\numer(w_2)=113,  &r_2:=\rem(w_2)=5.
    \end{array}
    \]
    We have $h'=\lcm(h_1,h_2)=1302$ and
    \begin{align*}
    r'_1 &\equiv r_1 \cdot \frac{h'}{h_1} \equiv 3 \cdot \frac{1302}{2} \equiv 3\pmod{-6},\\
    r'_2 &\equiv r_2 \cdot \frac{h'}{h_2} \equiv 5 \cdot \frac{1302}{651} \equiv 4\pmod{-6},\\
    m' &=\frac{n-1}{\gcd(n-1,r'_2-r'_1)}=\frac{-6}{\gcd(-6,1)}=6,\\
    r &\equiv r'_1 \cdot m' \equiv 3 \cdot 6 \equiv 0\pmod{-6},\\
    h &= h' \cdot m' = 1302 \cdot 6 = 7812.
    \end{align*}
Therefore, the closure of $W_2$ is determined by the join $(r_1,h_1) \vee (r_2,h_2) = (r,h) = (0,7812)$:
    \begin{align*}
        \rQuoth{0}{7812} &= \left\{ \frac{a}{7812 \cdot (-6)} \mathrel{}\middle|\mathrel{} -7812 \leq a \leq 39060 \text{ and } a \equiv 0\pmod{-6} \right\}\\ 
        &= \left\{ \frac{-39060}{46872}, \frac{-39054}{46872}, \dots, \frac{7812}{46872} \right\}  = \left\{ \frac{-6510}{7812}, \frac{-6509}{7812}, \dots, \frac{1302}{7812} \right\}.
    \end{align*}    
The closure $\cl(W_2)=\rWordh{0}{7812}$ consists of the reverses of the repetends of the $(-5)$-ary expansions of these fractions.
Thus $\cl(W_2)$ contains $7813$ words, and these yield $1329$ words up to cyclic shifts. Categorizing the elements of the closure by length, we get
\begin{itemize}
    \item 7680 words of length 6 (1280 words up to cyclic shifts),
    \item 120 words of length 3 (40 words up to cyclic shifts),
    \item 8 words of length 2 (4 words up to cyclic shifts) and
    \item 5 words of length 1 (all of the one-letter words over $\{0,1,2,3,4\}$).
\end{itemize}
Note that $\cl(W_1)\subseteq\cl(W_2)$, since if $w$ is the reverse of an element of $\cl(W_1)$, then $\fraction(w) \in \rQuoth{0}{7812}$ in base $n=-5$ (it suffices to verify this for one representative of each conjugacy class):
\begin{align*}      
        \fraction(1) &= \frac{\wval(1)}{(-5)^1-1} = \frac{1}{-6} = \frac{-1302}{7812},\\[2mm]
        \fraction(3) &= \frac{\wval(3)}{(-5)^1-1} = \frac{3}{-6} = \frac{-3906}{7812},\\[2mm]
        \fraction(20) &= \frac{\wval(20)}{(-5)^2-1} = \frac{-10}{24} = \frac{-3255}{7812},\\[2mm]
        \fraction(42) &= \frac{\wval(42)}{(-5)^2-1} = \frac{-18}{24} = \frac{-5859}{7812},\\[2mm]
        \fraction(122100) &= \frac{\wval(122100)}{(-5)^6-1} = \frac{-2100}{15624} = \frac{-1050}{7812},\\[2mm]
        \fraction(421300) &= \frac{\wval(421300)}{(-5)^6-1} = \frac{-11300}{15624} = \frac{-5650}{7812},\\[2mm]
        \fraction(231440) &= \frac{\wval(231440)}{(-5)^6-1} = \frac{-4420}{15624} = \frac{-2210}{7812},\\[2mm]
        \fraction(311410) &= \frac{\wval(311410)}{(-5)^6-1} = \frac{-8780}{15624} = \frac{-4390}{7812},\\[2mm]
        \fraction(341330) &= \frac{\wval(341330)}{(-5)^6-1} = \frac{-6940}{15624} = \frac{-3470}{7812},\\[2mm]
        \fraction(344322) &= \frac{\wval(344322)}{(-5)^6-1} = \frac{-7308}{15624} = \frac{-3654}{7812}.
    \end{align*}
\end{example}
As the next example shows, the closures of the same set of words with $n=5$ and $n=-5$ are not comparable in general. 

\begin{example}
    The closure of the set $\{220\}$ for $n=5$ consists of the reverses of the repetends of the 5-ary expansions of the fractions in 
    \[
    \rQuoth{0}{31}=\left\{\frac{0}{124},\frac{4}{124},\dots,\frac{124}{124}\right\}=\left\{\frac{0}{31},\frac{1}{31},\dots,\frac{31}{31}\right\}.
    \]
    This closure contains 2 one-letter, and 30 three-letter words (10 three-letter words up to conjugacy). Listing one representative from each conjugacy class, we get
    \[
    0,\,4,\,004,\,013,\,022,\,031,\,044,\,112,\,134,\,143,\,224,\,233.
    \]

    Now calculating the closure of the same set for $n=-5$, we get the set of fractions
    \[
    \rQuoth{4}{21}=\left\{\frac{-100}{126},\frac{-94}{126},\dots,\frac{20}{126}\right\}.
    \]
    The closure contains 21 three-letter words (7 three-letter words up to conjugacy). We can once again list a representative from each conjugacy class:
    \[
    004,\,013,\,022,\,031,\,112,\,244,\,334.
    \]
    The first closure contains 2 one-letter words, while the second contains none. On the other hand, the first closure does not contain the word $334$, which is an element of the second closure. Thus the two closures are incomparable.
\end{example}

\begin{remark}\label{rem:ureskrumpli}
    Given a pair $(r,h)$, there exists a word $w$ with $\rem(w)=r$ and $\denom(w)=h$ if and only if there is an integer $a$ such that $a \equiv r\pmod{n-1}$ and $\gcd(a,h)=1$. This is clear from the definition of the set $\rQuoth{r}{h}$, and the reader may wish to verify that these conditions are equivalent to $\gcd(r,h,n-1)=1$.
\end{remark}

\section{The special case $d_0=0$}\label{sec:d_0=0}
Let us consider those one-dimensional strong affine representations of $\mathcal{P}_{\abs{n}}$ that correspond to arithmetic sequences starting with $0$.
Thus we have $d_0=0$ and $D=\{0,h,2h,\dots,(\abs{n}-1)h\}$, where the difference $h \in \N$ is relatively prime to $n$. 
Recall that $r$ denotes the residue of $-d_0$ modulo $\abs{n-1}$, i.e., we have $r=0$ throughout this section. 

From the poset $\latn$ defined in Definition~\ref{def:latn} we only need the subposet $\{0\}\times\{h\in\N\colon\gcd(h,n)=1\}$.
It is easy to see that this set is just the principal filter ${\uparrow}(0,1) \coloneqq \{ (r,h) : (0,1) \leq (r,h)\}$; moreover, it is isomorphic to the poset of the sets $\rQuoth{0}{h}$, or, equivalently, to the poset of the sets $\rWordh{0}{h}$ ordered by inclusion:
\begin{equation}\label{e:order0}
    \rWordh{0}{h_1} \subseteq \rWordh{0}{h_2} \iff (0,h_1)\leq(0,h_2)\iff h_1\mid h_2.
\end{equation}
Thus ${\uparrow}(0,1)$ is isomorphic to the sublattice of the divisibility lattice of positive integers relatively prime to $n$.
In particular, the lattice operations are
\[(0,h_1)\wedge(0,h_2)=(0,\gcd(h_1,h_2))\qquad\text{and}\qquad(0,h_1)\vee(0,h_2)=(0,\lcm(h_1,h_2)).\]

Given a finite set $W$ of primitive words over the alphabet $\Sigma$, we are now interested in the least set of the form $\rWordh{0}{h}$ containing $W$.
This defines a closure operator $\clz$ analogous to $\cl$ (cf. Remark~\ref{rem:closureoperator}), and we have $\clz(W) \supseteq \cl(W)$ for all finite sets of primitive words $W \subseteq \mathcal{W}$.
To find $\clz(W)$, we need to determine the least set $\rQuoth{0}{h}$ such that each member of $W$ occurs among the reverses of the repetends of the fractions in $\rQuoth{0}{h}$.
We can simplify the conditions in Definition~\ref{d:QrhWrh} for the fractions in $\rQuoth{0}{h}$ in the following way:
\begin{align*}    
\rQuoth{0}{h}&=\left\{q\in I_{n} \ \middle|\   \exists a\in\Z\colon q=\frac{a}{h(n-1)}\text{ and } a\equiv 0\pmod{n-1}\right\}\\[2mm]
&=\left\{q\in I_{n} \ \middle|\  \exists b\in\Z\colon  q=\frac{b}{h}\right\}.
\end{align*}
Let $w\in \Sigma^{\ell}$ be a primitive word, and let us consider the corresponding fraction $\fraction(w)=\wval\big(\overleftarrow{w}\big)/(n^{\ell}-1)$.
We have $\fraction(w) \in \rQuoth{0}{h}$ if and only if the denominator of the reduced form of $\fraction(w)$ divides $h$:
\begin{equation}\label{e:Q0h}
    \fraction(w)\in\rQuoth{0}{h}\iff \left.\frac{ n^\ell-1}{\gcd\big( n^\ell-1,\wval\big(\overleftarrow{w}\big)\big)}\mathrel{}\right|\mathrel{}h.
\end{equation}

\begin{definition}
    For a primitive word $w\in \Sigma^{\ell}$, let
     \[
     \denom_0( w)=\frac{ n^\ell-1}{\gcd\big( n^\ell-1,\wval\big(\overleftarrow{w}\big)\big)}.
     \]    
\end{definition}

From \eqref{e:Q0h} we can immediately deduce the following formula for the closures of finite sets of primitive words:
\begin{equation}\label{e:cl0w}
    \clz(W) = \rWordh{0}{\lcm\{\denom_0(w) : w \in W\}}.
\end{equation}

The following theorem states that every nonempty finite set of words that is closed with respect to the closure operator $\clz$ (i.e., every closed set except for $\emptyset$ and $\mathcal{W}$) can be obtained as the closure of a single word.

\begin{theorem}\label{t:one-closure}
    For every positive integer $h$ that is relatively prime to $n$, there is a word $w \in \mathcal{W}$ such that $\clz(\{w\})=\rWordh{0}{h}$.
\end{theorem}
\begin{proof}
      We have seen that $\rQuoth{0}{h} = \{b/h : b \in \Z \text{ and } b/h \in I_n\}$.
      Since $I_n$ is of length one, there exist $h$ consecutive values of $b$ for which $b/h \in \rQuoth{0}{h}$.
      Among these numbers, there is at least one that is relatively prime to $h$.
      Let $b$ be such a value, i.e., $\gcd(b,h)=1$ and $b/h \in I_n$.
      Let $w$ be the reverse of the repetend of $b/h$, then we have $\fraction(w)=b/h$; moreover, this fraction is in reduced form.
      Since $\denom_0(w)$ is defined as the denominator of the reduced form of $\fraction(w)$, we have $h = \denom_0(w)$.
      From \eqref{e:cl0w} we can conclude that $\clz(\{w\}) = \rWordh{0}{\denom_0(w)}  = \rWordh{0}{h}$.
\end{proof}

We use the observations above to define a quasiorder on the set of primitve words such that the sets $\rWordh{0}{h}$ are exactly the principal ideals of this quasiordered set.

\begin{definition}
    Let $\lesssim$ be the binary relation on $\mathcal{W}$ defined by 
    $$ w_1\lesssim w_2\iff\denom_0( w_1)\mid\denom_0( w_2).$$
\end{definition}

\begin{remark}\label{r:lesssim equivalence}
    Observe that $\lesssim$ is a quasiorder on $\mathcal{W}$ and the equivalence relation corresponding to $\lesssim$ can be described as follows:
    $$ w_1 \sim  w_2 \iff w_1\lesssim w_2\text{ and } w_2\lesssim w_1\iff\denom_0( w_1)=\denom_0( w_2).$$
\end{remark}

\begin{definition}
    Let $(A,\lesssim)$ be a quasiordered set, and let $I\subseteq A$. Then $I$ is an \emph{ideal} of $(A,\lesssim)$ if the following hold:
    \begin{enumerate}[(1)]
        \item the set $I$ is not empty;
        \item for every $a\in I$ and $b\in A$, if $b\lesssim a$ then $b\in I$; 
        \item for every $a,b\in I$ there is $c\in I$ such that $a\lesssim c$ and $b\lesssim c$.
    \end{enumerate}
    For every $a\in A$, the set ${\downarrow} a=\{a'\colon a'\lesssim a\}$ is an ideal, called the \emph{principal ideal} generated by $a.$
\end{definition}
\begin{theorem}\label{t:principalideals}
    For any $W\subseteq\mathcal{W}$, there is a positive integer $h$ such that $\rWordh{0}{h}=W$ if and only if $W$ is a principal ideal in $(\mathcal{W};\lesssim)$.
    Consequently, the closed sets with respect to $\clz$ are exactly the principal ideals of $(\mathcal{W};\lesssim)$ together with $\emptyset$ and $\mathcal{W}$.

\end{theorem}
\begin{proof}
    From \eqref{e:Q0h} and \eqref{e:cl0w} it follows immediately that 
    \begin{equation}\label{e:down(w)=cl0(w)}
        \text{for every } w\in\mathcal{W}, \text{ we have } \clz(\{w\}) = {\downarrow} w \text{ in } (\mathcal{W}; \lesssim).
    \end{equation}
    If $W=\rWordh{0}{h}$ for some positive integer $h$, then Theorem~\ref{t:one-closure} guarantees the existence of a word $w$ such that $\clz(\{w\})=\rWordh{0}{h}$ for some word $w$; then \eqref{e:down(w)=cl0(w)} implies that $W={\downarrow} w$.
    
    Conversely, assume that $W$ is a principal ideal, i.e., $W = {\downarrow} w$ for some $w \in \mathcal{W}$. By \eqref{e:down(w)=cl0(w)}, we have $W = {\downarrow} w = \clz(\{w\})$, and \eqref{e:cl0w} shows that $\clz(\{w\}) = \rWordh{0}{h}$ with $h=\denom_0(w)$.
\end{proof}
\begin{remark}
    Each equivalence class with respect to $\sim$ is finite, and the induced poset $(\mathcal{W}/\sim;\leq)$ is isomorphic to a sublattice of $(\N;\mid)$.
    It is clear that in the latter lattice principal ideals are the same as finite ideals.
    Thus the same is true for $(\mathcal{W}/\sim;\leq)$, hence Theorem~\ref{t:principalideals} could be reformulated as follows:
    if $W \subseteq \mathcal{W}$ is a \emph{finite} set of words, then there is a positive integer $h$ such that $\rWordh{0}{h}=W$ if and only if $W$ is an ideal in $(\mathcal{W};\lesssim)$.
\end{remark}
\begin{remark}
    Considering the general case when $d_0$ is not assumed to be zero, it is possible to construct a similar binary relation $\lessapprox$ on $\mathcal{W}$ with the help of the poset $(\latn;\leq)$ and the functions $\denom(w)$ and $\rem(w)$:
    \[w_1\lessapprox w_2\iff \denom(w_1)\mid\denom(w_2)\text{ and }\rem(w_1)\cdot\frac{\denom(w_2)}{\denom(w_1)}\equiv\rem(w_2)\pmod{n-1}.\]
    By Proposition~\ref{p:leq} and Remark~\ref{rem:closureoperator}, we have $w_1\lessapprox w_2$ if and only if $\cl(\{w_1\}) \subseteq \cl(\{w_2\})$.
    It is then easy to see that ${\downarrow}w=\cl(\{w\})=\rWordh{\rem(w)}{\denom(w)}$ for every word $w\in\mathcal{W}$. 
    Thus the quasiorder $\lessapprox$ captures the closures of one-element sets.
    As we have seen in Theorem~\ref{t:one-closure}, in the case of the closure operator $\clz$ corresponding to arithmetic sequences starting with zero, all closed sets except for $\emptyset$ and $\mathcal{W}$ are of this form.
    However, this is not the case for the closure operator $\cl$ (i.e., for arbitrary arithmetic sequences): as the following example shows, not every set $\rWordh{r}{h}$ is a principal ideal in $(\mathcal{W};\lessapprox)$ if $n\neq 2$.
\end{remark}

\begin{example}
    By Remark~\ref{rem:ureskrumpli}, there is no word $w \in \mathcal{W}$ with $\denom(w)=n-1$ and $\rem(w)=0$ if $n \neq 2$.
    Let us verify without using Remark~\ref{rem:ureskrumpli} that $\rWordh{0}{n-1}$ is not the closure of any of its elements.
    We will use the notation $d=\abs{n}-1$ from Section~\ref{sec:expansions}.
    
    If $n > 2$, then $\rQuoth{0}{n-1}$ consists of the fractions $\frac{0}{n-1},\dots,\frac{n-1}{n-1}$ (see Example~\ref{ex:Q(0,n^l-1)} with $\ell=1$), and the corresponding words are $0,\dots,d$ (all words of length $1$ over $\Sigma$).
    The closures of these words are the following: $\cl(\{0\})=\cl(\{d\})=\{0,d\}$ and $\cl(\{c\})=\{c\}$ for $c=1,\dots,n-2$.
    
    If $n \leq -2$, then $\rQuoth{0}{n-1}$ consists of the fractions $\frac{-n}{n-1},\dots,\frac{-1}{n-1}$, and the corresponding words are $d0,d,\dots,1,0,0d$ (all words of length $1$ over $\Sigma$ and the words $d0$ and $0d$, which are conjugates of each other).
    The closures of these words are the following: $\cl(\{d0\})=\cl(\{0d\})=\{d0,0d\}$ and $\cl(\{c\})=\{c\}$ for $c=d,\dots,0$.
\end{example}

We conclude this section by discussing some results on the lengths of words in the equivalence classes of $(\mathcal{W};\sim)$.

\begin{theorem} Let $n$ be an integer with $\abs n\geq 2$.
    \begin{enumerate}[(i)]
        \item If $n\geq 2$ and $w\in\mathcal{W}$, then $\abs w=\ord_{\denom_0(w)}(n)$.
        \item If $n\leq -2$ and $w\in\mathcal{W}\setminus\{0d,d0\}$, then $\abs w=\ord_{\denom_0(w)}(n)$.
    \end{enumerate}
\end{theorem}
\begin{proof}
    If $w\in \mathcal{W}$ is of length $\ell$, then $\fraction(w)=\frac{\wval\left(\overleftarrow{w}\right)}{n^{\ell}-1}$ is the fraction that has $\overleftarrow{w}$ as its repetend. Reducing this fraction, the denominator becomes $\denom_0(w)$. Proposition~\ref{p:ordhn is fraction rep length} tells us that the repetend of this fraction is of length $\ord_{\denom_0(w)}(n)$ if $n$ is positive or $\overleftarrow{w}\notin\{0d,d0\}$.
    Since $\abs w=\abs{\overleftarrow{w}}$, the theorem is proved.
\end{proof}
\begin{corollary}\label{c:sim class length}
    If $n \geq 2$, then words within the same equivalence class with respect to $\sim$ have the same length.
    The same is true for $n \leq -2$ with the single exception of the equivalence class of $d0$, which consists of the words $d0$, $0d$ together with all one-letter words whose value is relatively prime to $n-1$:
    \begin{equation*}
    \big\{d0,0d\big\} \cup \big\{ c \in \Sigma : \gcd(c,n-1) = 1 \big\}.
    \end{equation*}
\end{corollary}

\begin{remark}
    It is easy to see that if $\cl(w_1) \subseteq \cl(w_2)$, then $\cl_0(w_1) \subseteq \cl_0(w_2)$. Thus $w_1\lessapprox w_2$ implies $w_1\lesssim w_2$, hence any $\approx$ class is a subset of a $\sim$ class. Moreover, if $n \leq -2$, then $\denom(0d)=\denom(d0)=1$, $\rem(0d)=\rem(d0)=-n$ and $\rWordh{-n}{1}=\{d0,0d\}$, thus $d0$ and $0d$ form a 2-element $\approx$ class. Taking Corollary~\ref{c:sim class length} into account, this means that if $w\approx w'$, then $\abs{w}=\abs{w'}$.
\end{remark}

\section{Concluding remarks and open problems}\label{sec:conclusion}

We have provided a description of the one-dimensional strong affine representations of the polycyclic monoids in the special case when the system of residues inducing the representation is an arithmetic sequence.
Since every pair of numbers is an arithmetic sequence, our results cover all one-dimensional strong affine representations of $\mathcal{P}_2$.
Let us just point out one interesting fact, a kind of converse of Corollary~\ref{c:sim class length}, in the case $n=2$.
\begin{proposition}
    If $n=2$ and $\ell \geq 2$, then all words of length $\ell$ form a single equivalence class with respect to $\sim$ if and only if $2^\ell-1$ is a Mersenne prime.
\end{proposition}
\begin{proof}
    Assume first that $2^\ell-1$ is a Mersenne prime.
    Then we have $\gcd(2^\ell-1,\wval(\overleftarrow{w}))=1$ for all words $w \in \mathcal{W}$ of length $\ell$.
    Therefore, these words all have the same $\denom_0(w)$ value, hence they are all equivalent with respect to $\sim$.
    On the other hand, Corollary~\ref{c:sim class length} shows that there are no other words in this equivalence class.
    
    Now let us assume that $2^\ell-1$ is a composite number.
    Then the smallest prime divisor $p$ of $2^\ell-1$ is at most $\sqrt{2^\ell-1}$, thus it has at most $\lceil \ell/2 \rceil$ digits in base $2$.
    Padding by zeros on the left, we obtain a word $w$ of length $\ell$ with $\wval(w) = p$.
    Since $w$ starts with at least $\lfloor \ell/2 \rfloor$ zeros, it must be a primitive word.
    However, $w'=0\dots01 \in \{0,1\}^\ell$ is also primitive, but $w \nsim w'$, as $\denom_0(w) = (2^\ell-1)/p$ and $\denom_0(w') = 2^\ell-1$.
\end{proof}

The general characterization of the one-dimensional strong affine representations of the polycyclic monoids is still unsolved. 
However, since the case $n=2$ was easier to handle, it seems natural to look at higher dimensions with $\abs{D}=2$.
This means that we consider a representation on $\Z^\nu$ given by maps of the form $f_i(\mathbf{x})=N\mathbf x +\mathbf d_i\ (i=0,1)$, where $N \in \Z^{\nu \times \nu}$ is a matrix with $\abs{\det(N)}=2$.

Restricting ourselves to one dimension again, we are also interested in comparing the sets of words obtained from negative and positive values of $n$. 
If we fix $k\geq 2$, are there sets of words that we can obtain with both $n=-k$ and $n=k$? 
Are there sets of words that we can only obtain in one case but not the other? 
Can we obtain a closure system by considering sets of words from both the negative and the positive case?

We fixed the indices in $D=\{d_0,\dots, d_{\abs n-1}\}$ such that $d_0<\dots<d_{\abs n -1}$, and this obviously influences the sets of words describing the representations. Removing this restriction, we have more freedom, thus more sets of words can occur, and it would be interesting to characterize these sets of words.
Relaxing the problem even further, one could delete the labels from the graph induced by the representations, and determine the arising graphs up to isomorphism. This means that instead of the sets of words, we only consider the multisets of the lengths of the words.

\bmhead{Acknowledgements}

The research of the authors was partially supported by the National Research, Development and Innovation Office of Hungary under grants K138892 and ADVANCED 153383.

\section*{Declarations}


\begin{itemize}
\item \textbf{Competing interests}: The authors have no conflicts of interest to declare that are relevant to the content of this article.
\end{itemize}

\bibliography{ODSARPM}

\end{document}